\documentclass[reqno, 11pt]{amsart}
\pdfoutput=1
\makeatletter
\let\origsection=\section \def\section{\@ifstar{\origsection*}{\mysection}} 
\def\mysection{\@startsection{section}{1}\z@{.7\linespacing\@plus\linespacing}{.5\linespacing}{\normalfont\scshape\centering\S}}
\makeatother        

\usepackage{amsmath,amssymb,amsthm}
\usepackage{amsfonts}
\usepackage{mathrsfs}
\usepackage{mathabx}\changenotsign
\usepackage{dsfont}

\RequirePackage{bm}
 
\usepackage{xcolor}

\usepackage[backref]{hyperref}
\hypersetup{
    colorlinks,
    linkcolor={red!60!black},
    citecolor={green!60!black},
    urlcolor={blue!60!black}
}

\usepackage{cleveref}

\usepackage{graphicx}
\usepackage{verbatim}
\usepackage[open,openlevel=2,atend]{bookmark}
\usepackage[abbrev,msc-links,backrefs]{amsrefs} 

\usepackage{doi}

\renewcommand{\PrintDOI}[1]{\doi{#1}}

\usepackage[T1]{fontenc}
\usepackage{lmodern}
\usepackage[babel]{microtype}
\usepackage[english]{babel}

\usepackage{geometry}

\numberwithin{equation}{section}
\numberwithin{figure}{section}

\usepackage[inline]{enumitem}

\let\polishlcross=\l
\def\l{\ifmmode\ell\else\polishlcross\fi}

\def\paragraph#1{%
  \noindent\textbf{#1.}\enspace}

\let\emptyset=\varnothing
\let\setminus=\smallsetminus

\makeatletter
\def\moverlay{\mathpalette\mov@rlay}
\def\mov@rlay#1#2{\leavevmode\vtop{   \baselineskip\z@skip \lineskiplimit-\maxdimen
   \ialign{\hfil$\m@th#1##$\hfil\cr#2\crcr}}}
\newcommand{\charfusion}[3][\mathord]{
    #1{\ifx#1\mathop\vphantom{#2}\fi
        \mathpalette\mov@rlay{#2\cr#3}
      }
    \ifx#1\mathop\expandafter\displaylimits\fi}
\makeatother

\DeclareFontFamily{U}  {MnSymbolC}{}
\DeclareSymbolFont{MnSyC}         {U}  {MnSymbolC}{m}{n}
\DeclareFontShape{U}{MnSymbolC}{m}{n}{
    <-6>  MnSymbolC5
   <6-7>  MnSymbolC6
   <7-8>  MnSymbolC7
   <8-9>  MnSymbolC8
   <9-10> MnSymbolC9
  <10-12> MnSymbolC10
  <12->   MnSymbolC12}{}
\DeclareMathSymbol{\powerset}{\mathord}{MnSyC}{180}

\theoremstyle{plain}
\newtheorem{thm}{Theorem}[section]
\newtheorem{theorem}[thm]{Theorem}
\newtheorem{lemma}[thm]{Lemma}
\newtheorem{corollary}[thm]{Corollary}
\newtheorem{proposition}[thm]{Proposition}
\newtheorem{problem}[thm]{Problem}
\newtheorem{conjecture}[thm]{Conjecture}
\newtheorem{question}[thm]{Question}
\newtheorem{thm-intro}{Theorem}[]
\newtheorem*{claim*}{Claim}
\newtheorem{claim}{Claim}[]

\theoremstyle{definition}
\newtheorem{definition}[thm]{Definition}

\newtheorem{observation}[thm]{Observation}

\usepackage{accents}

\usepackage[symbol]{footmisc}

\newcommand{\no}[1]{}

\newcommand*{\Bidirected}[1]{\overset{\text{\tiny$\bm\leftrightarrow$}}{#1}}

\newcommand{\GB}{$\mathsf{GB}$}

\newcommand*{\Cut}[2]{\Fkt{\partial_{#1}}{#2}}

\newcommand*{\Fkt}[2]{#1\!\left( #2 \right)}

\newcommand*{\BigO}[1]{\mathcal{O}(#1)}

\author[Heuer, Steiner, Wiederrecht]{Karl Heuer \and Raphael Steiner \and Sebastian Wiederrecht}

\address[Heuer, Wiederrecht]{Technische Universit\"{a}t Berlin,
Institute of Software Engineering and Theoretical Computer Science,
Ernst-Reuter-Platz 7, 10587 Berlin, Germany}

\address[Steiner]{Technische Universit\"{a}t Berlin, Institute of Mathematics, Straße des 17. Juni 136, 10623 Berlin, Germany}

\email{\tt karl.heuer@tu-berlin.de}
\email{\tt steiner@math.tu-berlin.de}
\email{\tt sebastian.wiederrecht@tu-berlin.de}

\title[Even circuits in oriented matroids]{Even circuits in oriented matroids}

\begin{document}

\begin{abstract}
In this paper we generalise the even directed cycle problem, which asks whether a given digraph contains a directed cycle of even length, to orientations of regular matroids.
We define non-even oriented matroids generalising non-even digraphs, which played a central role in resolving the computational complexity of the even dicycle problem.
Then we show that the problem of detecting an even directed circuit in a regular matroid is polynomially equivalent to the recognition of non-even oriented matroids.
Our main result is a precise characterisation of the class of non-even oriented bond matroids in terms of forbidden minors, which complements an existing characterisation of non-even oriented graphic matroids by Seymour and Thomassen.
\end{abstract}

\maketitle

% 05B35 Combinatorial aspects of matroids and geometric lattices

% 05C20: Directed graphs (digraphs), tournaments

% 05C70 : Edge subsets with special properties (factorization, matching, partitioning, covering and packing, etc.)

% 05C75 : Structural characterization of families of graphs

% 05C83 : Graph minors

% 05C85 : Graph algorithms (graph-theoretic aspects)

% 52C40 : Oriented matroids in discrete geometry

\section{Introduction}
\label{sec:intro}

Deciding whether a given digraph contains a directed cycle, briefly \textit{dicycle}, of even length is a fundamental problem for digraphs and often referred to as the \textit{even dicycle problem}.
The computational complexity of this problem was unknown for a long time and several polynomial time equivalent problems have been found~\cites{klee1984signsolvability, manber1986digraphs, thomassen1986sign, mccuaig2004polya}.
The question about the computational complexity was resolved by Robertson, Seymour and Thomas~\cite{robertson1999permanents} and independently by McCuaig~\cite{mccuaig2004polya} who stated polynomial time algorithms for one of the polynomially equivalent problems, and hence also for the even dicycle problem.

One of these polynomially equivalent problems makes use of the following definition.

\begin{definition}[\cite{seymour1987characterization}]
	Let $D$ be a digraph.
	We call $D$ \emph{non-even}, if there exists a set $J$ of directed edges in $D$ such that every directed cycle $C$ in $D$ intersects $J$ in an odd number of edges.
	If such a set does not exist, we call $D$ \emph{even}.
\end{definition}

Seymour and Thomassen proved that the decision problem whether a given digraph is non-even, is polynomially equivalent to the even dicycle problem.

\begin{theorem}[\cite{seymour1987characterization}]\label{seymthomequivalence}
	The problem of deciding whether a given digraph contains an even directed cycle, and the problem of deciding whether a given digraph is non-even, are polynomially equivalent.
\end{theorem}

Furthermore, Seymour and Thomassen~\cite{seymour1987characterization} characterised being non-even in terms of forbidden subgraphs.
Their result can be stated more compactly by formulating it in terms of forbidden butterfly minors, which is a commonly used notion in directed graph structure theory~\cites{johnson2001directed, packingdicycles, digridthm}, instead of forbidden subgraphs.
Before we state their result, let us define the notion of butterfly minors and fix another notation.

Given a digraph $D$, an edge $e \in E(D)$ is called \emph{butterfly-contractible} if it is not a loop and if it is either the unique edge emanating from its tail or the unique edge entering its head.
A \emph{butterfly minor} (sometimes also called \emph{digraph minor} or just \emph{minor}) of a digraph $D$ is any digraph obtained from $D$ by a finite sequence of edge-deletions, vertex-deletions and contractions of butterfly-contractible edges.

Note that the main idea behind the concept of a butterfly-contractible edge $e$ within a digraph $D$ is that every directed cycle in $D/e$ either equals one in $D$ or induces one in $D$ by incorporating~$e$.
This property does not necessarily hold if arbitrary edges are contracted.

For every $k \ge 3 $ let $\Bidirected{C}_k$ denote the symmetrically oriented cycle of length $k$ (also called \emph{bicycle}), i.\@e.\@ the digraph obtained from $C_k$ be replacing every edge by a pair of anti-parallel directed edges.

Now we can state the result of Seymour and Thomassen as follows.

\begin{theorem}[\cite{seymour1987characterization}]\label{thm:seymthomtheorem}
	A digraph $D$ is non-even if and only if no butterfly minor of $D$ is isomorphic to $\Bidirected{C}_k$ for some odd $k$.
\end{theorem}

The main purpose of this work is to lift the even dicycle problem to oriented matroids, and to extend \Cref{seymthomequivalence} and partially \Cref{thm:seymthomtheorem} to oriented matroids as well.

\subsection{The Even Directed Circuit Problem in Oriented Matroids}\label{subsec:evendirectedcircuitproblem}
${}$
\vspace{6pt} \newline
\indent In this paper we view a matroid as a tuple $M=(E,\mathcal{C})$ consisting of a finite \emph{ground set} $E(M):=  E$ containing the \emph{elements} of $M$ and the family $\mathcal{C}$ of \emph{circuits} of $M$. 

In what follows we introduce a generalisation of the graph theoretic notion of being non-even to oriented matroids and state the main results of this work.
For our purposes, the most important examples of matroids are \emph{graphical matroids} and \emph{bond matroids}. 

Let $G=(V,E)$ be a graph.
The \emph{graphical matroid of $G$}, denoted by $M(G)$, is the matroid $(E,\mathcal{C})$ where the set $\mathcal{C}$ of circuits consists of all edge-sets of the cycles of $G$.
Analogously, the \emph{bond matroid of $G$} is $M^\ast(G)=(E,\mathcal{S})$ where $\mathcal{S}$ is the set of bonds (or minimal non-empty edge cuts) of $G$.
Note that $M(G)$ and $M^\ast(G)$ are the dual matroids of each another.

A matroid is called a \emph{graphic matroid}, resp.~a \emph{bond matroid} (also called \emph{cographic matroid}) if it is, respectively, isomorphic to the graphical or the bond matroid of some graph.

Digraphs can be seen as a special case of oriented matroids\footnote{For a formal and in-depth introduction of terms and notation used here please see \Cref{subsec:prelim}.} in the sense that every digraph $D$ has an associated oriented graphic matroid $M(D)$ whose signed circuits resemble the oriented cycles in the digraph $D$.
In this spirit, it is natural to lift questions concerning cycles in directed graphs to more general problems on circuits in oriented matroids.
The following algorithmic problem is the straight forward generalisation of the even dicycle problem to oriented matroids, and the main motivation of the paper at hand.
\begin{problem}\label{evencircuit1}
	Given an oriented matroid $\vec{M}$, decide whether there exists a directed circuit of even size in $\vec{M}$.
\end{problem}

Our first contribution is to generalise the definition of non-even digraphs to oriented regular matroids in the following sense.

\begin{definition}
	Let $\vec{M}$ be an oriented matroid.
	We call $\vec{M}$ \emph{non-even} if its underlying matroid is regular and there exists a set $J \subseteq E(\vec{M})$ of elements such that every directed circuit in $\vec{M}$ intersects $J$ in an odd number of elements. If such a set does not exist, we call $\vec{M}$ even.
\end{definition}

The reader might wonder why the preceding definition concerns only regular matroids.
This has several reasons.
The main reason is a classical result by Bland and Las Vergnas~\cite{orientability} which states that a binary matroid is orientable if and only if it is regular.
Hence, if we were to extend the analysis of non-even oriented matroids beyond the regular case, we would have to deal with orientations of matroids which are not representable over $\mathbb{F}_2$.
This has several disadvantages, most importantly that cycles bases, which constitute an important tool in all of our results, are not guaranteed to exist any more.
Furthermore, some of our proofs make use of the strong orthogonality property of oriented regular matroids\footnote{For a definition we refer to \Cref{subsec:prelim}}, which fails for non-binary oriented matroids.
Lastly, since \Cref{evencircuit1} is an algorithmic questions, oriented regular matroids have the additional advantage that they allow for a compact encoding in terms of totally unimodular matrices, which is not a given for general oriented matroids. 

The first result of this article is a generalisation of \Cref{seymthomequivalence} to oriented matroids as follows:

\begin{theorem}
\label{thm:oddjoinsandevencuts}
The problems of deciding whether an oriented regular matroid represented by a totally unimodular matrix contains an even directed circuit, and the problem of recognising whether an oriented regular matroid given by a totally unimodular matrix is non-even, are polynomially equivalent.
\end{theorem}

\Cref{thm:oddjoinsandevencuts} motivates a structural study of the class of non-even oriented matroids, as in many cases the design of a recognition algorithm for a class of objects is based on a good structural understanding of the class.
In order to state our main result, which is a generalisation of \Cref{thm:seymthomtheorem} to graphic and cographic oriented matroids, we have to introduce a new minor concept.
We naturally generalise the concept of butterfly minors to regular oriented matroids, in the form of so-called \emph{generalised butterfly minors}.
\begin{definition}
	Let $\vec{M}$ be an orientation of a regular matroid $M$.
	An element $e \in E(\vec{M})$ is called \emph{butterfly-contractible} if there exists a cocircuit $S$ in $M$ such that $(S\setminus\{e\},\{e\})$ forms a signed cocircuit of $\vec{M}$.\footnote{For a definition of a signed (co)circuits see~\Cref{subsec:prelim}.}
	A \emph{generalised butterfly minor (\GB-minor for short)} of $\vec{M}$ is any oriented matroid obtained from $\vec{M}$ by a finite sequence of element deletions and contractions of butterfly-contractible elements (in arbitrary order).
\end{definition}

Note that the generalised butterfly-contraction captures the same fundamental idea as the initial one for digraphs while being more general:
Given a butterfly-contractible element $e$ of a regular oriented matroid~$\vec{M}$, we cannot have a directed circuit $C$ of $\vec{M} / e$ such that $(C, \{e\})$ is a signed circuit of $\vec{M}$ \footnote{In this case, $(C, \{e\})$ together with a signed cocircuit $(S\setminus\{e\},\{e\})$ would contradict the orthogonality property (see \Cref{subsec:prelim}, (\ref{equ:ortho})) for oriented matroids.}, and hence either $C$ or $C \cup \{e\}$ must form a directed circuit of $\vec{M}$.

Replacing the notion of butterfly minors by \GB-minors allows us to translate \Cref{thm:seymthomtheorem} to the setting of oriented matroids in the following way:

\begin{proposition}
\label{thm:forbiddengraphicminors}
An oriented graphic matroid $\vec{M}$ is non-even if and only if none of its \GB-minors is isomorphic to $M(\Bidirected{C}_k)$ for some odd $k \ge 3$.
\end{proposition}

As our main result, we complement \Cref{thm:forbiddengraphicminors} by determining the list of forbidden \GB-minors for cographic non-even oriented matroids.
We need the following notation:
For integers $m, n \ge 1$ we denote by $\vec{K}_{m,n}$ the digraph obtained from the complete bipartite graph $K_{m,n}$ by orienting all edges from the partition set of size $m$ towards the partition set of size $n$.

\begin{theorem}
\label{thm:forbiddencographicminors}
An oriented bond matroid $\vec{M}$ is non-even if and only if none of its \GB-minors is isomorphic to $M^\ast(\vec{K}_{m,n})$ for any $m, n \ge 2$ such that $m+n$ is odd.
\end{theorem}

To prove \Cref{thm:forbiddencographicminors} we study those digraphs whose oriented bond matroids are non-even.
Equivalently, these are the digraphs admitting an \emph{odd dijoin}, which is an edge set hitting every directed bond an odd number of times.
After translating \GB-minors into a corresponding minor concept on directed graphs, which we call \emph{cut minors}\footnote{See the beginning of \Cref{sec:dijoin} for a precise definition.}, we show that the class of digraphs with an odd dijoin is described by two infinite families of forbidden cut minors (\Cref{thm:forbiddenminors}).
Finally, we translate this result to oriented bond matroids in order to obtain a proof of \Cref{thm:forbiddencographicminors}.

The structure of this paper is as follows.
In \Cref{subsec:prelim} we introduce the needed notation and basic facts about digraphs, matroids and oriented matroids for this paper.
Furthermore, we prove that non-even oriented matroids are closed under \GB-minors (\Cref{prop:GBminorcontainment}), which is then used to prove \Cref{thm:forbiddengraphicminors} in the same section.
We start \Cref{sec:even-circuit complexity} by showing that the even directed circuit problem for general oriented matroids cannot be solved using only polynomially many calls to a signed circuit oracle~(\Cref{prop:oracleisexponential}).
The remainder of the section is devoted to the proof of \Cref{thm:oddjoinsandevencuts}. We also note that \emph{odd} directed circuits can be detected in polynomial time in orientations of regular matroids~(\Cref{prop:oddcycleproblem}).
In \Cref{sec:dijoin} we characterise those digraphs that admit an odd dijoin (\Cref{thm:forbiddenminors}) and use this to deduce our main result, \Cref{thm:forbiddencographicminors}.

\section{Background}\label{subsec:prelim}

This section is dedicated to a formal introduction of basic terms and notation used throughout this paper.
However, we assume basic familiarity with digraphs and matroid theory.
For basic notation and facts about digraphs we refer the reader to \cite{bang-jensen}.
For missing terminology and basic facts from matroid theory not mentioned or mentioned without proof in the following, please consult the standard reading \cites{oxley, welsh}.

For two sets $X, Y$ we denote by $X+Y:= (X \cup Y)\setminus (X \cap Y)$ their \emph{symmetric difference}. For $n \in \mathbb{N}$ we denote $[n]:=\{1,2,\ldots,n\}$.

\subsection*{(Di)graphs}

Graphs considered in this paper are multi-graphs and may include loops. 
Digraphs may have loops and multiple (parallel and anti-parallel) directed edges (sometimes called \emph{edges}).
Given a digraph $D$, we denote by $V(D)$ its vertex set and by $E(D)$ the set of directed edges.
A directed edge with tail $u \in V(D)$ and head $v \in V(D)$ is denoted by $(u,v)$ if this does not lead to confusion with potential parallel edges.
By $U(D)$ we denote the \emph{underlying multi-graph} of $D$, which is the undirected multi-graph obtained from $D$ by forgetting the orientations of the edges.
Given a digraph $D$ and a partition $(X,Y)$ of its vertex set, the set $D[X,Y]$ of edges with one endpoint in $X$ and one endpoint in $Y$, if it is non-empty, is referred to as a \emph{cut}. 
A cut of $D$ is called \emph{minimal} or a \emph{bond}, if there is no other cut of $D$ properly contained in it. 
If there is no edge of $D$ with head in $X$ and tail in $Y$, the cut $D[X,Y]$ is called \emph{directed} and denoted by $\Cut{}{X}$ (the set of edges leaving $X$).
A \emph{dijoin} in a digraph is a set of edges intersecting every directed cut (resp.~every directed bond).

\subsection*{Matroids}

Matroids can be used to represent several algebraic and combinatorial structures of dependencies.
The so-called \emph{linear or representable} matroids are induced by vector configurations in linear spaces.
Let $V=\mathbb{F}^n$ be a vector-space over a field $\mathbb{F}$ and let $X=\{x_1,\ldots,x_k\} \subseteq V$ for some $k \in \mathbb{N}$.
Let $A$ be the $n\times k$-matrix over $\mathbb{F}$ whose columns are $x_1,\ldots,x_k$.
Then we define the \emph{column matroid induced by $A$} as  $M[A]:= (\{x_1,\ldots,x_k\},\mathcal{C}_A)$, where its set of circuits $\mathcal{C}_A$ consists of the inclusion-wise minimal collections of linearly dependent vectors from $\{x_1,\ldots,x_k\}$.
It is a well-known fact that $M[A]$ is indeed a matroid for any choice of a matrix $A$.
A matroid $M$ is called \emph{$\mathbb{F}$-linear} or \emph{representable over the field $\mathbb{F}$} if there is a matrix $A$ with entries in $\mathbb{F}$ such that $M \simeq M[A]$.
Graphic matroids and bond matroids, as introduced in \Cref{subsec:evendirectedcircuitproblem}, form part of a larger class, the so-called \emph{regular matroids}.
A matroid $M$ is called regular if it is $\mathbb{F}$-linear for every field $\mathbb{F}$.
The following equivalent characterisation of regular matroids is useful for encoding purposes.
A matrix with entries in $\mathbb{R}$ is called \emph{totally unimodular} if every square submatrix has determinant $-1, 0$ or $1$.

\begin{theorem}[\cite{tutteregularchar}]\label{tutteregularity}
Let $M$ be a matroid. Then $M$ is regular if and only if $M \simeq M[A]$ for a totally unimodular real-valued matrix $A$. Furthermore, for any field $\mathbb{F}$, reinterpreting the $\{-1,0,1\}$-entries of $A$ as elements of $\mathbb{F}$, we obtain an $\mathbb{F}$-linear representation of $M$.
\end{theorem}

Every graphic and every bond matroid is regular, but not vice-versa. 
Regular matroids are in turn generalised by the \emph{binary matroids}, which are the $\mathbb{F}_2$-linear matroids.

We conclude this paragraph with the important notion of \emph{matroid minors}, which generalises the concept of minors in graph theory.
Given a matroid $M$ and an element $e \in E(M)$, we denote by $M-e$ and $M/e$ the matroids obtained from $M$ by \emph{deleting} and \emph{contracting}~$e$.

These operations are consistent with deletions and contractions in graph theory in the following sense:
If $G$ is a graph and $e \in E(G)$, let us denote by $G/e$ the graph obtained by contracting the edge $e$ and by $G-e$ the graph obtained by deleting $e$.
Then it holds that ${M(G/e) \simeq M(G)/e}$, $M(G-e) \simeq M(G)-e, M^\ast(G-e)=M^\ast(G)/e$, and finally $M^\ast(G/e) \simeq M^\ast(G)-e$.

\subsection*{Oriented Matroids}

For missing terminology and basic facts from the theory of oriented matroids not mentioned or mentioned without proof in the following, please consult the standard reading~\cite{bibel}. 

An \emph{oriented matroid} $\vec{M}$ is a tuple $(E,\mathcal{C})$ consisting of a ground set $E$ of elements and a collection $\mathcal{C}$ of \emph{signed subsets} of $E$, i.\@e.\@ ordered partitions $(C^+,C^-)$ of subsets $C$ of $E$ into \emph{positive} and \emph{negative} parts such that the following axioms are satisfied:
\begin{itemize}
\item $(\emptyset,\emptyset) \notin \mathcal{C}$
\item If $(C^+,C^-) \in \mathcal{C}$, then $(C^-,C^+) \in \mathcal{C}$.
\item If $(C_1^+,C_1^-), (C_2^+,C_2^-) \in \mathcal{C}$ such that $C_1^+ \cup C_1^- \subseteq C_2^+ \cup C_2^-$, then one of the equations ${(C_1^+,C_1^-) = (C_2^+,C_2^-)}$ or ${(C_1^+,C_1^-) = (C_2^-,C_2^+)}$ holds.
    \item Let $(C_1^+,C_1^-), (C_2^+,C_2^-) \in \mathcal{C}$ such that $(C_1^+,C_1^-) \neq (C_2^-,C_2^+)$, and let ${e \in C_1^+ \cap C_2^-}$.
    Then there is some $(C^+,C^-) \in \mathcal{C}$ such that $C^+ \subseteq (C_1^+ \cup C_2^+) \setminus \{e\}$ and ${C^- \subseteq (C_1^- \cup C_2^-) \setminus \{e\}}$.
\end{itemize}
In case these axioms are satisfied, the elements of $\mathcal{C}$ are called \emph{signed circuits}.

Two oriented matroids $\vec{M}_1=(E_1,\mathcal{C}_1)$ and $\vec{M}_2=(E_2,\mathcal{C}_2)$ are called \emph{isomorphic} if there exists a bijection $\sigma:E_1 \rightarrow E_2$ such that $\{(\sigma(C^+),\sigma(C^-)) \; | \; (C^+,C^-) \in \mathcal{C}_1\}=\mathcal{C}_2$. 
For every oriented matroid $\vec{M}=(E,\mathcal{C})$ and a signed circuit $X=(C^+,C^-) \in \mathcal{C}$, we denote by $\underline{X}:= C^+ \cup C^-$ the so-called \emph{support} of $X$.
From the axioms for signed circuits it follows that the set family $\underline{\mathcal{C}}:= \{\underline{X} \; | \; X \in \mathcal{C}\}$ over the ground set $E$ defines a matroid $M=(E,\underline{\mathcal{C}})$, which we refer to as the \emph{underlying matroid} of $\vec{M}$, and vice versa, $\vec{M}$ is called an \emph{orientation} of $M$.
A matroid is called \emph{orientable} if it admits at least one orientation. 
A signed circuit $(C^+,C^-)$ is called \emph{directed} if either $C^+=\emptyset$ or $C^-=\emptyset$.
We use this definition also for the circuits of the underlying matroid $M$, i.\@e.\@, a circuit of $M$ is \emph{directed} in $\vec{M}$ if $(C,\emptyset)$ (or equivalently $(\emptyset,C)$) is a directed signed circuit of $\vec{M}$.
We say that $\vec{M}$ is \emph{totally cyclic} if every element of $M$ is contained in a directed circuit, and \emph{acyclic} if there exists no directed circuit. 

Classical examples of oriented matroids can be derived from vector configurations in real-valued vector spaces and, most importantly for the investigations in this paper, from directed graphs. 

Given a configuration $\{x_1,x_2,\ldots,x_k\} \in \mathbb{R}^n$ of vectors for some $k \in \mathbb{N}$, consider the matroid $M[A]$ with $A=(x_1,x_2,\ldots,x_k) \in \mathbb{R}^{n \times k}$.
Given a circuit $C=\{x_{i_1},x_{i_2},\ldots,x_{i_\ell}\} \in \mathcal{C}$, then there are scalars $\alpha_1,\ldots,\alpha_k \in \mathbb{R}^k\setminus \{\mathbf{0}\}$ such that $\sum_{\ell=1}^{k}{\alpha_ix_{i_\ell}}=0$, and the coefficients $\alpha_i$ are determined up to multiplication with a common scalar.
It is therefore natural to assign two signed sets to the circuit as follows: $X(C):= (C^+,C^-)$ and $-X(C):= (C^-,C^+)$, where $C^+:= \{x_{i_\ell} \; | \; \alpha_{i_\ell}>0\}$ and $C^-:= \{x_{i_\ell} \; | \; \alpha_{i_\ell}<0\}$.
The \emph{oriented matroid induced by $A$} is then defined as $\vec{M}[A]=(\{x_1,\ldots,x_k\},\{X(C),-X(C) \; | \; C \in \mathcal{C}_A\})$.

Given a digraph $D$ we can, as in the undirected case, associate with it two different kinds of oriented matroids with ground set $E(D)$.
Unsurprisingly, their underlying matroids are exactly the graphical respectively the bond matroid of $U(D)$.

\begin{definition}
Let $D$ be a digraph. 
\begin{itemize}
    \item For every cycle $C$ in $D$, let $(C^+,C^-),(C^-,C^+)$ be the two tuples describing a partition of $E(C)$ into sets of forward and backward edges, according to some choice of cyclical traversal of $C$. Then $\{(C^+,C^-),(C^-,C^+) \; | \; C \text{ cycle in }D\}$ forms the set of signed circuits of an orientation $M(D)$ of $M(U(D))$, called the \emph{oriented graphic matroid induced by $D$}.
    \item For every bond $S=D[X,Y]$ in $D$, let $S^+$ be the set of edges in $S$ with tail in $X$ and head in $Y$, and let $S^-$ contain those edges on $S$ with tail in $Y$ and head in $X$. Then $\{(S^+,S^-), (S^-, S^+) \; | \; S\text{ is a bond in }D\}$ forms the set of signed circuits of an orientation $M^\ast(D)$ of $M^\ast(U(D))$, called the \emph{oriented bond matroid induced by $D$}.
\end{itemize}
\end{definition}

Note that the directed circuits of an oriented graphic matroid are exactly the edge-sets of the directed cycles of the corresponding digraph $D$.
Similarly, the directed circuits in an oriented bond-matroid are the edge-sets of the directed bonds in the corresponding digraph.

Given an oriented matroid $\vec{M}=(E,\mathcal{C})$ and an element $e \in E(\vec{M})$, we denote by $\vec{M}-e$ and $\vec{M}/e$ the matroids obtained from $\vec{M}$ by \emph{deleting} and \emph{contracting} $e$, respectively.
The signed circuits of these matroids are defined as follows:
\begin{align*}
\mathcal{C}(M-e)&:= \{(C^+,C^-) \in \mathcal{C} \; | \; e \notin C^+ \cup C^-\},
\end{align*}
\begin{align*}
\mathcal{C}(M/e)&:= \{(C^+\setminus \{e\},C^-\setminus \{e\}) \; | \; (C^+,C^-) \in \mathcal{C}, \not\exists X \in \mathcal{C}: \underline{X}-e \subsetneq (C^+ \cup C^-)-e\}\setminus \{(\emptyset,\emptyset)\}.
\end{align*}
These definitions generalize to subsets $Z \subseteq E(\vec{M})$, here we denote by $\vec{M}-Z$ resp.~$\vec{M}/Z$ the oriented matroids obtained from $\vec{M}$ by successively deleting (resp.~contracting) all elements of $Z$ (in arbitrary order\footnote{It is well/known that the order in which elements are deleted resp.~contracted does not affect the outcome of the process.}). We also use the notation $\vec{M}[Z]=\vec{M}-(E(\vec{M}-Z))$.

Again, in the case of graphic and cographic oriented matroids, the deletion and contraction operations resemble the same operations in directed graphs:
Given a digraph $D$ and $e \in E(D)$, denote by $D/e$ the digraph obtained by deleting $e$ and identifying the endpoints of $e$.
We then have $M(D)-e \simeq M(D-e), M(D)/e \simeq M(D/e)$ and $M^\ast(D)-e \simeq M^\ast(D)/e, M^\ast(D)/e \simeq M^\ast(D-e)$.

For an oriented matroid $\vec{M}$ with a collection $\mathcal{C}$ of signed circuits, let $\mathcal{S}$ be defined as the set of signed vectors $(S^+,S^-)$ satisfying the following \emph{orthogonality property} for every signed circuit $C=(C^+,C^-) \in \mathcal{C}$:
\begin{equation*}\label{equ:ortho}
    (S^+ \cap C^+) \cup (S^- \cap C^-) \neq \emptyset \Longleftrightarrow (S^+ \cap C^-) \cup (S^- \cap C^+) \neq \emptyset. \tag{$\ast$}
\end{equation*}
Then $\mathcal{S}$ is called the set of \emph{signed cocircuits} of $\vec{M}$.
The supports of the signed cocircuits form exactly the cocircuits of the underlying matroid $M$.
A signed cocircuit $(S^+,S^-)$ is called \emph{directed} if $S^+=\emptyset$ or $S^-=\emptyset$.
If the underlying matroid $M$ of $\vec{M}$ is regular, then the following stronger orthogonality holds for every singed circuit $(C^+,C^-) \in \mathcal{C}$, and every signed cocircuit $(S^+,S^-) \in \mathcal{S}$:
\begin{equation*}\label{equ:strongortho}
    |C^+ \cap S^+|+|C^- \cap S^-| = |C^+ \cap S^-| + |C^- \cap S^+|. \tag{$\ast \ast$}
\end{equation*}
For any digraph $D$ the signed cocircuits of $M(D)$ are the same as the signed circuits of $M^\ast(D)$, while the signed cocircuits of $M^\ast(D)$ are exactly the signed circuits of $M(D)$.

We conclude this first part of the preliminary section by stating a couple of important facts concerning orientations of (regular) matroids from the literature.
\begin{theorem}[\cite{bibel}]\label{regularityOM}
Let $\vec{M}$ be an orientation of a regular matroid $M$. Then there exists a totally unimodular matrix $A$ such that $\vec{M} \simeq \vec{M}[A]$ and $M \simeq M[A]$.
\end{theorem}

We will also need the following matroidal version of the famous Farkas' Lemma:
\begin{theorem}[\cite{bibel}]\label{lemma:farkas}
Let $\vec{M}$ be an oriented matroid and $e \in E(M)$. Then $e$ is contained in a directed circuit of $\vec{M}$ if and only if it is not contained in a directed cocircuit.
\end{theorem}

\subsection{Non-Evenness and \GB-minors}\hfill\\
\indent Our main result, \Cref{thm:forbiddencographicminors}, builds on the important fact that the non-even oriented matroids are closed under the \GB-minor relation.
In this subsection we present a proof of this fact and use it to derive \Cref{thm:forbiddengraphicminors} from \Cref{thm:seymthomtheorem}.

\begin{lemma}\label{prop:GBminorcontainment}
	Every \GB-minor of a non-even oriented matroid is non-even.
\end{lemma}

\begin{proof}
	It suffices to show the following two statements:
	For every non-even oriented matroid $\vec{M}$ and every element $e \in E(\vec{M})$, the oriented matroid $\vec{M}-e$ is non-even as well, and for every element $e \in E(\vec{M})$ which is butterfly-contractible, the oriented matroid $\vec{M}/e$ is non-even as well.
	The claim then follows by repeatedly applying these two statements.
	
	For the first claim, note that since the underlying matroid $M$ of $\vec{M}$ is regular, so is the underlying matroid of $\vec{M}-e$.
	Let $J \subseteq E(\vec{M})$ be a set of elements intersecting every directed circuit in $\vec{M}$ an odd number of times.
	Then clearly the set $J \setminus \{e\}$ intersects every directed circuit in $\vec{M}-e$ an odd number of times, proving that $\vec{M}-e$ is non-even.
	
	For the second claim, let $e \in E(\vec{M})$ be butterfly-contractible.
	Let $S$ be a cocircuit of $M$ such that $(S\setminus \{e\},\{e\})$ forms a signed cocircuit of $\vec{M}$.
	Then the underlying matroid of $\vec{M}/e$ is a matroid minor of the regular matroid $M$ and is hence regular.
	Define $J' \subseteq E(\vec{M}) \setminus \{ e \}$ via
\begin{align*}
J' := 
\begin{cases}
J  & \mbox{if } e \notin J \\
J+S  & \mbox{if } e \in J.
\end{cases}
\end{align*}
	We claim that for every directed circuit $C$ in $\vec{M}/e$, the intersection $C \cap J'$ is odd.
	Indeed, by definition either $C$ is a directed circuit also in $\vec{M}$ not containing $e$, or $C \cup \{e\}$ is a directed circuit in $\vec{M}$, or $(C,\{e\})$ is a signed circuit of $\vec{M}$.
	The last case however is impossible, as then the signed circuit $(C,\{e\})$ and the signed cocircuit $(S\setminus \{e\},\{e\})$ in $\vec{M}$ would yield a contradiction to the orthogonality (\ref{equ:ortho}) of oriented matroids.
	
	In the first case, since $e \notin C$, we must have $S \cap C=\emptyset$ as otherwise again $C$ and the signed cocircuit $(S\setminus \{e\},\{e\})$ form a contradiction to the orthogonality property (\ref{equ:ortho}).
	This then shows that indeed $|C \cap J'|=|C \cap (J' \setminus S)|=|C \cap (J \setminus S)|=|C \cap J|$ is odd, as required.
	
	In the second case, the orthogonality property (\ref{equ:strongortho}) of regular oriented matroids applied with the directed circuit $C \cup \{e\}$ and the signed cocircuit $(\{e\},S\setminus \{e\})$ within $\vec{M}$ yield that the equation ${|(C \cup \{ e \}) \cap (S \setminus \{e\})| = |(C \cup \{ e \}) \cap \{e\}|=1}$ holds.
	So let $C \cap S = \{f\}$ for some element $f \in E(\vec{M}) \setminus \{ e \}$.
	By definition of $J'$, if $e \notin J$, then we have ${|C \cap J'|=|C \cap J|=|(C \cup \{e\}) \cap J|}$, which is odd.
	If $e \in J$, then we have (modulo $2$) 
	\begin{align*}
	|C \cap J'|=|C \cap (J+S)|=|(C \cap J)+(C \cap S)| \equiv |C \cap J|+|\{f\}| = |(C \cup \{e\}) \cap J|,
	\end{align*}
	which is odd.
	Hence, we have shown that $|C \cap J'|$ is odd in every case, which yields that $\vec{M}/e$ is a non-even oriented matroid.
	This concludes the proof.
\end{proof}

\Cref{prop:GBminorcontainment} allows us to immediately prove the correctness of \Cref{thm:forbiddengraphicminors}.

\begin{proof}[Proof of \Cref{thm:forbiddengraphicminors}]
	We prove both directions of the equivalence.
	Suppose first that $\vec{M}$ is non-even.
	Then by \Cref{prop:GBminorcontainment} every oriented matroid isomorphic to a \GB-minor of $\vec{M}$ is non-even as well.
	Hence it suffices to observe that none of the matroids $M(\Bidirected{C}_k)$ for odd $k \ge 3$ is non-even.
	However, this follows directly since any element set $J$ in $M(\Bidirected{C}_k)$ intersecting every directed circuit an odd number of times corresponds to an edge set in $\Bidirected{C}_k$ intersecting every directed cycle an odd number of times, which cannot exist since by \Cref{thm:seymthomtheorem} none of the digraphs $\Bidirected{C}_k$ is non-even for an odd $k \ge 3$.
	
	Vice versa, suppose that no \GB-minor of $\vec{M}$ is isomorphic to $M(\Bidirected{C}_k)$ for any odd $k \ge 3$.
	Let $D$ be a digraph such that $\vec{M} \simeq M(D)$.
	We claim that $D$ must be non-even. Suppose not, then by \Cref{thm:seymthomtheorem} $D$ admits a butterfly minor isomorphic to $\Bidirected{C}_k$ for some odd $k \ge 3$.
	We now claim that $M(D)$ has a \GB-minor isomorphic to $M(\Bidirected{C}_k)$.
	For this, it evidently suffices to verify the following general statement:
	
	If an edge $e$ of a digraph $F$ is butterfly-contractible in $F$, then within $M(F)$ the corresponding element $e$ of $M(F)$ is butterfly-contractible.
	
	Indeed, let $e=(u,v)$ for distinct vertices $u, v \in V(D)$.
	Then by definition either $u$ has out-degree $1$ or $v$ has in-degree $1$ in $D$.
	In the first case, $e$ is the unique edge leaving $u$ in the cut $D[\{u\},V(D)\setminus \{u\}]$, while in the second case $e$ is the only edge entering $v$ in the cut $D[V(D)\setminus \{v\},\{v\}]$.
	Since every cut is an edge-disjoint union of bonds, we can find in both cases a bond containing $e$ where $e$ is the only edge directed away resp.~towards the side of the bond that contains $u$ resp.~$v$.

	Since the oriented bonds in $D$ yield the signed cocircuits of $M(D)$, this shows that there is a cocircuit $S$ in $M(D)$ such that $(S\setminus \{e\},\{e\})$ is a signed cocircuit.
	Hence, $e$ is a butterfly-contractible element of $M(D)$.
	This shows that $M(\Bidirected{C}_k)$ is isomorphic to a \GB-minor of $M(D) \simeq \vec{M}$ which contradicts our initial assumption that no \GB-minor of $\vec{M}$ is isomorphic to $M(\Bidirected{C}_k)$.
	Hence, $D$ is non-even, and there exists $J \subseteq E(D)$ such that every directed cycle in $D$ contains an odd number of edges from $J$.
	The same set $J$ also certifies that $\vec{M} \simeq M(D)$ is non-even, and this concludes the proof of the equivalence.
\end{proof}

\section{On the Complexity of the Even Directed Circuit Problem}
\label{sec:even-circuit complexity}

The formulation of \Cref{evencircuit1} is rather vague, as it is not clear by which means the oriented matroid $\vec{M}$ is given as an input to an algorithm designed for solving the problem, and in which way we will measure its efficiency.
For the latter, it is natural to aim for an algorithm which performs a polynomial number of elementary steps in terms of the number of elements of $\vec{M}$.
This also resembles the even dicycle problem in digraphs, where we aim to find an algorithm running in polynomial time in $|E(D)|$.

For the former, it is not immediately clear how to encode the (oriented) matroid, and hence how to make information contained in the (oriented) matroid available to the algorithm.
For instance the list of all circuits of a matroid, if given as input to an algorithm, will usually have exponential size in the number of elements, and therefore disqualify as a good reference value for efficiency of the algorithm.
For that reason, different computational models (and efficiency measures) for algorithmic problems in matroids (see \cite{oraclesmatroids}) and oriented matroids (see \cite{oraclesOMs}) have been proposed in the literature.
These models are based on the concept of \emph{oracles}.
For a family $\mathcal{F} \subseteq 2^{E(M)}$ of objects characterising the matroid $M$, an \emph{oracle} is a function $f:2^{E(M)} \rightarrow \{\texttt{true}, \texttt{false}\}$ assigning to every subset a truth value indicating whether or not the set is contained in $\mathcal{F}$.
If $\mathcal{F}$ for instance corresponds to the collection of circuits, cocircuits, independent sets, or bases of a matroid, we speak of a \emph{circuit-, cocircuit-, independence-}, or \emph{basis-oracle}.
 Similarly, for oriented matroids we can define several oracles~\cite{oraclesOMs}.
 Maybe the most natural choice for an oriented matroid-oracle for \Cref{evencircuit1} is the \emph{circuit oracle}, which given any subset of the element set together with a $\{+, -\}$-signing of its elements, reveals whether or not this signed subset forms a signed circuit of the oriented matroid.
 This computational model applied to \Cref{evencircuit1} yields the following question.

\begin{question}\label{circuitoracle}
	Does there exist an algorithm which, given an oriented matroid $\vec{M}$, decides whether there exists a directed circuit in $\vec{M}$ of even size, by calling the circuit-oracle of $\vec{M}$ only $\BigO{|E(\vec{M})|^c}$ times for some $c\in\mathbb{N}$?
\end{question}

However, as it turns out, the answer to the above problem is easily seen to be negative, even when the input oriented matroid $\vec{M}$ is graphic. 

\begin{proposition}\label{prop:oracleisexponential}
	Any algorithm deciding whether a given oriented graphic matroid on $n$ elements, for some $n \in \mathbb{N}$, contains an even directed circuit must use at least $2^{n-1}-1$ calls to the circuit-oracle for some instances.
\end{proposition}

\begin{proof}
Suppose towards a contradiction there was an algorithm which decides whether a given oriented graphic matroid contains an even directed circuit and uses at most $2^{n-1}-2$ oracle calls for any input oriented graphic matroid on elements $E := \{1,\ldots,n\}$.
Now, playing the role of the oracle, we will answer all of the (at most $2^{n-1}-2$) calls of the algorithm by \texttt{false}.
Since there are exactly $2^{n-1}-1$ non-empty sets $Y \in 2^E$ of even size, there must be an even non-empty subset $Y$ of $E$ such that the algorithm did not call the oracle with any input signed set whose support is $Y$.
But this means the algorithm cannot distinguish between the oriented graphic matroids $(E,\mathcal{C}_0)$ and $(E,\mathcal{C}_Y)$, where $\mathcal{C}_0:= \emptyset$ and $\mathcal{C}_Y:= \{(Y,\emptyset),(\emptyset, Y)\}$, which result in the same oracle-answers to the calls by the algorithm, while $(E,\mathcal{C}_0)$ contains no even directed circuit, but $(E,\mathcal{C}_Y)$ does.
This shows that the algorithm does not work correctly, and this contradiction proves the assertion.
\end{proof}

The above result and its proof give a hint that maybe in general the use of oriented matroid-oracles to measure the efficiency of algorithms solving \Cref{evencircuit1} is doomed to fail.
One should therefore look for a different encoding of the input oriented matroids in order to obtain a sensible algorithmic problem.
In this paper, we solve this issue by restricting the class of possible input oriented matroids to \emph{oriented regular matroids}, which allow for a much simpler and compact encoding via their representation by totally unimodular matrices (cf.\@ \Cref{tutteregularity,regularityOM}).
The following finally is the actual algorithmic problem we are going to discuss in this paper.

\begin{problem}\label{therealproblem}
Is there an algorithm which decides, given as input a totally unimodular matrix $A \in \mathbb{R}^{m \times n}$ for some $m, n \in \mathbb{N}$, whether $\vec{M}[A]$ contains an even directed circuit, and runs in time polynomial in $mn$?
\end{problem}

The alert reader might be wondering what happens if in the above problem we aim to detect odd instead of even directed circuits. The reason why this problem is not a center of study in our paper is that it admits a simple polynomial time solution, which is given in the form of \Cref{prop:oddcycleproblem} at the end of this section.

The next statement translates the main results from~\cite{robertson1999permanents} and~\cite{mccuaig2004polya} to our setting to show that \Cref{therealproblem} has a positive answer if we restrict to graphic oriented matroids as inputs.

\begin{lemma}
There exists an algorithm which, given as input any totally unimodular matrix $A \in \mathbb{R}^{m \times n}$ for some $m, n \in \mathbb{N}$ such that $\vec{M}[A]$ is a graphic oriented matroid, decides whether $\vec{M}[A]$ contains a directed circuit of even size, and which runs in time polynomial in $mn$.
\end{lemma}

\begin{proof}
The main results of Robertson et al.\@~\cite{robertson1999permanents} and McCuaig~\cite{mccuaig2004polya} yield polynomial time algorithms which, given as input a digraph $D$ (by its vertex- and edge-list) returns whether or not $D$ contains an even directed cycle.
Therefore, given a totally unimodular matrix $A \in \mathbb{R}^{m \times n}$ such that $\vec{M}[A]$ is graphic, if we can construct in polynomial time in $mn$ a digraph $D$ such that $\vec{M}[A] \simeq M(D)$, then we can decide whether $\vec{M[A]}$ contains a directed circuit of even size by testing whether $D$ contains an even directed cycle using the algorithms from~\cites{mccuaig2004polya, robertson1999permanents}.
Such a digraph can be found as follows:

First, we consider the \emph{unoriented} matroid $M[A]$ defined by the matrix $A$, which is graphic.
It follows from a result of Seymour~\cite{recoggraphmatroids} that using a polynomial number in $|E(M[A])|=n$ of calls to an independence-oracle for $M[A]$, we can compute a connected graph $G$ with $n$ edges such that $M(G) \simeq M[A]$.
For every given subset of columns of $A$ we can test linear independence in polynomial time in $mn$, and hence we can execute the steps of Seymour's algorithm in polynomial time.
Since $M(G) \simeq M[A]$, there must exist an orientation of $M(G)$ isomorphic to $\vec{M}[A]$, and this orientation in turn can be realized as $M(D)$ where $D$ is an orientation of $G$\footnote{The fact that every orientation of $M(G)$ can be realised as $M(D)$ for an orientation $D$ of $G$ follows from a classical result by Bland and LasVergnas~\cite{orientability}, who show that regular matroids (and particularly graphic ones) have a unique reorientation class.}.
To find the desired orientation $D$ of $G$ in polynomial time, we first compute a decomposition of $G$ into its blocks $G_1,\ldots,G_k$ (maximal connected subgraphs without cutvertices).

Next we (arbitrarily) select for every $i \in \{1,\ldots,k\}$ a special `reference'-edge $e_i \in E(G_i)$.
Note that two different orientations of $G$ obtained from each other by reversing all edges in one block result in the same oriented matroid, as cycles in $G$ are always entirely contained in one block.
Hence for every $i \in \{1,\ldots,k\}$ we can orient $e_i$ arbitrarily and assume w.\@l.\@o.\@g.\@ that this orientation coincides with the orientation in $D$.
Note that every block of $G$ which is not $2$-connected must be a $K_2$ forming a bridge in $G$.
In this case, the only edge of the block is our chosen reference-edge and already correctly oriented.
Now, for every $i \in \{1,\ldots,k\}$ such that $G_i$ is $2$-connected and every edge $e\in E(G_i)\setminus \{e_i\}$ there is a cycle $C$ in $G_i$ containing both $e_i$ and $e$.
This cycle can be computed in polynomial time using a disjoint-paths algorithm between the endpoints of $e$ and $e_i$.
Now we consider the minimally linearly dependent set of columns in $A$ corresponding to $C$, and compute the coefficients of a non-trivial linear combination resulting in $0$.
As we already know the orientation of $e_i \in E(C)$, this yields us the orientations of all edges on the cycle $C$ in $D$ and hence of the edge $e$.
In this way, we can compute all orientations of edges in $D$ in polynomial time in $mn$ and find the digraph $D$ such that $\vec{M}[A] \simeq M(D)$.
As discussed above, this concludes the proof.
\end{proof}

\subsection{Proof of \Cref{thm:oddjoinsandevencuts}}\label{sec:equivalence}\hfill\\
\indent We prepare the proof by a set of useful definitions and lemmas dealing with circuit bases of regular matroids.

\begin{definition}
Let $M$ be a binary matroid.
The \emph{circuit space} of $M$ is the $\mathbb{F}_2$-linear vector space generated by the incidence vectors $\mathbf{1}_C \in \mathbb{F}_2^{E(M)}$ defined by $\mathbf{1}_C(e):= 1$ for $e \in C$ and $\mathbf{1}_C(e):= 0$ for $e \notin C$ and all circuits $C$ of $M$.
A \emph{circuit basis} of $M$ is a set of circuits of $M$ whose incidence vectors form a basis of the circuit space.
Equivalently, we can consider the circuit space as a $\mathbb{F}_2$-linear subspace of the vector space whose elements are all the subsets of $E$ and where the sum $X+Y$ of two sets $X, Y \subseteq E(M)$ is defined as their symmetric difference.
\end{definition}

\begin{definition}
Let $\vec{M}$ be a regular oriented matroid and $M$ be its underlying regular matroid.
We call a circuit basis $\mathcal{B}$ of $M$ \emph{directed} if all elements of $\mathcal{B}$ are directed circuits of~$\vec{M}$.
\end{definition}

The next proposition is a well-known fact about the circuit space of a binary matroid.

\begin{proposition}\label{proposition:basesize}
Let $M$ be a binary matroid. Then the dimension of the circuit space of $M$ equals $|E(M)|-r(M)$.
\end{proposition}

The following lemma is crucial for the proof of \Cref{thm:oddjoinsandevencuts} as well as for our work on digraphs in \Cref{sec:dijoin}.

\begin{lemma}\label{lemma:circuitspace}
Let $\vec{M}$ be an oriented regular matroid.
If $\vec{M}$ is totally cyclic, then the underlying matroid $M$ admits a directed circuit basis.
Furthermore, for every coindependent set $A$ in $M$ such that $\vec{M}-A$ is totally cyclic, there exists a directed circuit basis of $M$ such that every $a \in A$ is contained in exactly one circuit of the basis.
\end{lemma}

\begin{proof}
We start by proving the first assertion concerning the existence of a directed circuit basis of $M$.
We use induction on $|E(M)|$.
If $M$ consists of a single element, the claim holds trivially, since every circuit is a loop and thus directed.
So assume now that $|E(M)|=k \ge 2$ and that the statement of the lemma holds for all oriented regular matroids on at most $k-1$ elements. Choose some $e \in E(M)$ arbitrarily.
Since $\vec{M}$ is totally cyclic, there exists a directed circuit $C_e$ containing $e$. 
Let us now consider the oriented regular matroid $\vec{M}-e$.
If $\vec{M}-e$ is totally cyclic, then we can apply the induction hypothesis to $\vec{M}-e$ and find a directed circuit basis $\mathcal{B}^-$ of $M-e$.
Now consider the collection $\mathcal{B} = \mathcal{B}^- \cup \{ C_e \}$ of directed circuits in $\vec{M}$.
The incidence vectors of these circuits are linearly independent over $\mathbb{F}_2$, as $C_e$ is the only circuit yielding a non-zero entry at element $e$.
Furthermore, we get by induction that $ |\mathcal{B}| = |E(M)|-1-r(M-e)+1=|E(M)|-r(M-e)=|E(M)|-r(M)$.
The last equality holds since $e$ is contained in the circuit $C_e$ and hence is not a coloop.
As this matches the dimension of the circuit space of $M$, we have found a directed circuit basis of $M$, proving the inductive claim.

It remains to prove the case where $\vec{M}-e$ is not totally cyclic, i.\@e.\@, there is an element not contained in a directed circuit.
By Farkas' Lemma (\Cref{lemma:farkas}) applied to $\vec{M}-e$ and this element there exists a directed cocircuit $S$ in $\vec{M}-e$. Then either $(S, \emptyset), (S \cup \{e\},\emptyset)$ or $(S, \{e\})$ form a signed cocircuit of $\vec{M}$.
Since $\vec{M}$ is totally cyclic, it contains no directed cocircuits, and hence only the latter case is possible, $(S,\{e\})$ must form a signed cocircuit.

Let us now consider the oriented regular matroid $\vec{M}/e$.
Since $\vec{M}$ is totally cyclic, so is $\vec{M}/e$.
By the induction hypothesis there exists a directed circuit basis $\mathcal{B}^-$ of $M/e$.
By definition, for every directed circuit $C \in \mathcal{B}^-$, either $C$ is a directed circuit in $\vec{M}$ not containing $e$, or $C \cup \{e\}$ is a directed circuit in $\vec{M}$, or $(C,\{e\})$ forms a signed circuit of $\vec{M}$.
The latter is however impossible, as in this case we can consider the signed cocircuit $X=(S,\{e\})$ and the signed circuit $Y=(C,\{e\})$ of $\vec{M}$, which satisfy $e \in X^- \cap Y^- \neq \emptyset$ but furthermore ${(X^+ \cap Y^-) \cup (X^- \cap Y^+)=\emptyset}$, violating the orthogonality property (\ref{equ:ortho}) of oriented matroids.

Hence, the set $\mathcal{B}:= \{C \; | \; C \in \mathcal{B}^-\text{ circuit in }M\}\cup \{C \cup \{e\} \; | \; C \in \mathcal{B}^-, C \cup \{e\}\text{ circuit in }M\}$ consists of $|\mathcal{B}|=|\mathcal{B}^-|=|E(M)|-1-r(M/e)=|E(M)|-r(M)$ many circuits of $M$ which are all directed ones in $\vec{M}$.
Note that for the last equality we used that $e$ is not a loop, as it is contained in the cocircuit $S \cup \{e\}$ of $M$.
Finally, we claim that the binary incidence vectors of the elements of $\mathcal{B}$ in $\mathbb{F}_2^{E(M)}$ are linearly independent.
This follows since the restriction of these vectors to the coordinates $E(M) \setminus \{ e \}$ equals the characteristic vectors of the elements of $\mathcal{B}^-$, which form a circuit basis of $M/e$.
This shows that we have found a directed circuit basis of $M$, proving the inductive claim.

For the second assertion, let a coindependent set $A$ in $M$ be given and suppose that $\vec{M}-A$ is totally cyclic.
We claim that for every $a \in A$ there exists a directed circuit $C_a$ in $\vec{M}$ such that $C_a \cap A=\{a\}$.
Equivalently, we may show that the oriented matroid $\vec{M}-(A\setminus \{a\})$ has a directed circuit containing $a$.
Towards a contradiction, suppose not, then by Farkas' Lemma (\Cref{lemma:farkas}) there exists a directed cocircuit $S$ in $\vec{M}-(A\setminus \{a\})$ containing $a$. Since $A$ is coindependent, $a$ is not a coloop of $M$ and hence $S \setminus \{a\} \neq \emptyset$.
Every directed circuit in $\vec{M}-(A\setminus \{a\})$ must be disjoint from $S$, and hence no $f \in S\setminus \{a\}$ is contained in a directed circuit of $\vec{M}-A$, contradicting our assumption that $\vec{M}-A$ is totally cyclic.
It follows that for each $a \in A$ a directed circuit $C_a$ with $C_a \cap A=\{a\}$ exists.

Next we apply the first assertion of this lemma to the totally cyclic oriented matroid $\vec{M}-A$.
We find that there exists a directed circuit basis $\mathcal{B}_A$ of $M-A$.
We claim that $\mathcal{B}:= \mathcal{B}_A \cup \{C_a \; | \; a \in A\}$ forms a directed circuit basis of $M$ satisfying the properties claimed in this lemma.
Indeed, every circuit in $\mathcal{B}$ is a directed circuit of $\vec{M}$, and for every $a \in A$ the circuit $C_a$ is the only circuit in $\mathcal{B}$ containing $a$.
Since the characteristic vectors of the elements of $\mathcal{B}_A$ are linearly independent as $\mathcal{B}_A$ is a circuit basis of $M-A$, we already get that the characteristic vectors of elements of $\mathcal{B}$ are linearly independent using that the characteristic vector of $C_a$ is the only basis-vector having a non-zero entry at the position corresponding to element~$a$.
To show that $\mathcal{B}$ indeed is a circuit basis of $M$, it remains to verify that it has the required size.
We have $|\mathcal{B}|=|A|+|\mathcal{B}_A|=|A|+|E(M-A)|-r(M-A)=|E(M)|-r(M)$, where for the latter equality we used that $r(M-A)=r(M)$ since $A$ is coindependent.
This concludes the proof of the second assertion.
\end{proof}

In order to prove our next lemma, we need the following result, which was already used by Seymour and Thomassen.

\begin{lemma}[\cite{seymour1987characterization}, Prop.~3.2] \label{lemma:oddandeven}
Let $E$ be a finite set and $\mathcal{F}$ a family of subsets of $E$. Then precisely one of the following statements holds:
\begin{enumerate}[label=(\roman*)]
\item There is a subset $J \subseteq E$ such that $|F \cap J|$ is odd for every $F \in \mathcal{F}$.
\item There are sets $F_1,\ldots,F_k \in \mathcal{F}$, where $k \in \mathbb{N}$ is odd, such that $\sum_{i=1}^{k}{F_i}=\emptyset$. 
\end{enumerate}
\end{lemma}

Please note that (i) and (ii) cannot hold simultaneously because if $k$ is odd and $F_1,\dotsc,F_k$ all have odd intersection with $A$, then the symmetric difference $\sum_{i=1}^k F_i$ has odd intersection with $A$.

We now derive the following corollary for totally cyclic oriented regular matroids by using \Cref{lemma:circuitspace} and applying \Cref{lemma:oddandeven} to a directed circuit basis.

\begin{corollary}\label{cor:matroidweighting}
Let $\vec{M}$ be a totally cyclic oriented regular matroid, and let $\mathcal{B}$ be a directed circuit basis of $M$.
Then there exists $J \subseteq E(\vec{M})$ such that $|C \cap J|$ is odd for every $C \in \mathcal{B}$. 
\end{corollary}
\begin{proof}
The claim is that $(i)$ in \Cref{lemma:oddandeven} with $E=E(\vec{M})$ and $\mathcal{F}:=\mathcal{B}$ holds true, so it suffices to rule out $(ii)$.
However, the latter would contradict the linear independence of the basis $\mathcal{B}$.
\end{proof}

Building on this corollary we derive equivalent properties for an oriented matroid to be non-even.

\begin{proposition} \label{proposition:matroidequivalence}
Let $\vec{M}$ be a totally cyclic oriented regular matroid and let $\mathcal B$ be a directed circuit basis of $M$.
Furthermore, let $J \subseteq E(M)$ be such that $|C \cap J|$ is odd for all $C \in \mathcal B$.
 Then the following statements are equivalent:
 \begin{enumerate}[label=(\roman*)]
  \item $\vec{M}$ is non-even.
  \item If $C_1,\ldots,C_k$ are directed circuits of $\vec{M}$ where $k \in \mathbb{N}$ is odd, then $\sum_{i=1}^k C_i \neq \emptyset$.
  \item Every directed circuit of $\vec{M}$ is a sum of an odd number of elements of $\mathcal{B}$.
  \item $|C \cap J|$ is odd for all directed circuits $C$ of $\vec{M}$.
 \end{enumerate}
\end{proposition}

\begin{proof}
\noindent
 \begin{description} 
  \item[``(i) $\Rightarrow$ (ii)''] This follows from \Cref{lemma:oddandeven} applied to the set of all directed circuits of $\vec{M}$.
  \item[``(ii) $\Rightarrow$ (iii)''] Let $C$ be a directed circuit of $\vec{M}$.
  Since $\mathcal B$ is a circuit basis of $M$, we can write $C = \sum_{i=1}^k C_i$ for some $k \in \mathbb{N}$ and $C_1,\ldots,C_k \in \mathcal B$.
  If $k$ were even, then the sum $C + \sum_{i=1}^k C_i = \emptyset$ would yield a contradiction to (ii).
  \item[``(iii) $\Rightarrow$ (iv)''] Let $C$ be a directed circuit of $\vec{M}$.
  By assumption, $C = \sum_{i=1}^k C_i$ with $C_1,\ldots,C_k \in \mathcal B$ and $k \in \mathbb{N}$ being odd.
  Since $J$ has odd intersection with all $C_i$, the set~$J$ has also odd intersection with $C$.
  \item[``(iv) $\Rightarrow$ (i)''] This implication follows directly from the definition of non-even. 
 \end{description}
\end{proof}

Before we turn towards the proof of \Cref{thm:oddjoinsandevencuts} we need the following result, yielding a computational version of Farkas' lemma for oriented regular matroids. Although we suspect the statement is well-known among experts, we include a proof for the sake of completeness.

\begin{lemma}\label{lemma:farkas-compute}
There exists an algorithm that, given as input a regular oriented matroid $\vec{M}$ represented by a totally unimodular matrix $A \in \{-1,0,1\}^{m \times n}$ such that $\vec{M} \simeq \vec{M}[A]$, and an element $e \in E(\vec{M})$, outputs either a directed circuit of $\vec{M}$ containing $e$ or a directed cocircuit of $\vec{M}$ containing $e$, and which runs in polynomial time in $mn$.
\end{lemma}

\begin{proof}
We first observe that we can decide in polynomial time in $mn$ whether $e$ is contained in a directed circuit or in a directed cocircuit of $\vec{M}$ (by Farkas' Lemma, we know that exactly one of these two options must be satisfied).
Let us denote for every element $f \in E(\vec{M})$ by $x_f \in \{-1,0,1\}^{m}$ the corresponding column-vector of $A$.
We need the following claim:

The element $e$ is contained in a directed circuit of $\vec{M}$ if and only if there exist non-negative scalars $\alpha_f \ge 0$ for $f \in E(\vec{M}) \setminus \{e\}$ such that $-x_e=\sum_{f \in E(\vec{M} \setminus \{e\})}{\alpha_f x_f}$. 

The necessity of this condition follows directly by definition of $\vec{M}[A]$: If $e$ is contained in a directed circuit with elements $e,f_1,\ldots,f_k$, then there are coefficients $\beta_e>0$ and $\beta_i>0$ for $1 \le i \le k$ such that $\beta_e x_e+\sum_{i=1}^{k}{\beta_i x_{f_i}}=0$, i.\@e.\@, $-x_e=\sum_{i=1}^{k}{\frac{\beta_i}{\beta_e}x_{f_i}}$.
On the other hand, if $-x_e$ is contained in the conical hull of $\{x_f|f \in E(\vec{M})\setminus\{e\}\}$, then we can select an inclusion-wise minimal subset $F \subseteq E(\vec{M} \setminus \{e\})$ such that $-x_e$ is contained in the conical hull of $\{x_f|f \in F\}$.
We claim that $\{e\} \cup F$ forms a directed circuit of $\vec{M}$.
By definition of $F$, it suffices to verify that the vectors $x_e$ and $x_f$ for $f \in F$ are minimally linearly dependent.
However, this follows directly by Carath\'{e}odory's Theorem: The dimension of the subspace spanned by $\{x_f|f \in F\}$ equals $|F|$, for otherwise we could select a subset of at most $|F|-1$ elements from $\{x_f|f \in F\}$ whose conical hull also contains $-x_e$, contradicting the minimality of $F$.
This shows the equivalence claimed above.

We can now use the well-known linear programming algorithm for linear programs with integral constraints by Khachiyan~\cites{khachiyan, gacslovasz} to decide in strongly polynomial time\footnote{Here we use the fact that all coefficients appearing in the linear system are $-1$,$0$ or $1$.} (and hence in polynomial time in $mn$) the feasibility of the linear inequality system
\begin{align*}
\sum_{f \in E(\vec{M} \setminus \{e\})}{\alpha_f x_f}=-x_e, \text{ with }  \alpha_f \ge 0.
\end{align*}
Therefore, we have shown that we can decide in polynomial time in $mn$ whether or not $e$ is contained in a directed circuit of $\vec{M}$.
Next we give an algorithm which, given that $e$ is contained in a directed circuit of $\vec{M}$, finds such a circuit in polynomial time: 

During the procedure, we update a subset $Z \subseteq E(\vec{M})$, which maintains the property that  it contains a directed circuit including $e$.
At the end of the procedure $Z$ will form such a directed circuit of $\vec{M}$. We initialise $Z:= E(\vec{M})$.
During each step of the procedure, we go through the elements $f \in Z \setminus \{e\}$ one by one and apply the above algorithm to test whether $\vec{M}[Z]-f$ contains a directed circuit including $e$.
At the first moment such an element is found, we put $Z:= Z\setminus \{f\}$ and repeat.
If no such element is found, we stop and output $Z$. 

Since we reduce the size of the set $Z$ at each round of the procedure, the above algorithm runs in at most $n$ rounds and calls the above decision algorithm for the existence of a directed circuit including $e$ at most $n-1$ times in every round.
All in all, the algorithm runs in time polynomial in $mn$.
It is obvious that the procedure maintains the property that $Z$ contains a directed circuit including $e$ and that at the end of the procedure all elements of $Z$ must be contained in this circuit, i.\@e.\@, $Z$ forms a directed circuit with the desired properties.  

To complete the proof we now give an algorithm which finds either a directed circuit or a directed cocircuit through a given element $e$ of $\vec{M}$ as follows:
First we apply the first (decision) algorithm, which either tells us that $e$ is contained in a directed circuit of $\vec{M}$, in which case we apply the second (detection) algorithm to find such a circuit.
Otherwise we know that $e$ is contained in a directed cocircuit of $\vec{M}$, in which case we compute in polynomial time a totally unimodular representing matrix $A^\ast$ with at most $n$ rows and $n$ columns\footnote{To find such a representing matrix, one can use Gaussian elimination to compute a basis $\mathcal{B}$ of $\text{ker}(A)$.
Since $A$ is totally unimodular, the vectors in $\mathcal{B}$ can be taken to be $\{-1,0,1\}$-vectors such that the matrix $A^\ast$ consisting of the elements of $\mathcal{B}$ written as row-vectors is totally unimodular as well.
It then follows from the orthogonality property of regular oriented matroids that $A^\ast$ indeed forms a representation of $\vec{M}^\ast$, using the fact that the row spaces of $A$ and $A^\ast$ are orthogonal complements.} of the dual regular oriented matroid $\vec{M}^\ast$.
As we know that $e$ is included in a directed circuit of $\vec{M}^\ast$, we can apply the second (detection) algorithm to $A^\ast$ and $\vec{M}^\ast$ instead of $A$ and $\vec{M}$ in order to find a directed cocircuit in $\vec{M}$ containing $e$ in polynomial time.
\end{proof}

Given a regular oriented matroid $\vec{M}$ we shall denote by $TC(\vec{M})$ the largest totally cyclic deletion minor of $\vec{M}$, i.\@e.\@ the deletion minor of $\vec{M}$ whose ground set is
\begin{align*}
E(TC(\vec{M})) :=  \bigcup\{C \; | \; C \textnormal{ is a directed circuit of } \vec{M} \}. 
\end{align*}
From~\Cref{lemma:farkas-compute} we directly have the following.

\begin{corollary}\label{cor:totallycyclic}
Let $\vec{M}$ be a regular oriented matroid represented by a totally unimodular matrix $A \in \{-1,0,1\}^{m \times n}$ for some $m \in \mathbb{N}$ and $n = |E(M)|$.
Then we can compute a totally unimodular representative matrix of $TC(\vec{M})$ in time polynomial in $mn$. \qed
\end{corollary}

The last ingredient we shall need for the proof of \Cref{thm:oddjoinsandevencuts} is a computational version of the first statement of \Cref{lemma:circuitspace} combined with \Cref{cor:matroidweighting}.

\begin{lemma}\label{lemma:compute_base_and_cover}
Let $\vec{M}$ be a totally cyclic regular oriented matroid represented by a totally unimodular matrix $A \in \{-1,0,1\}^{m \times n}$ for some $m \in \mathbb{N}$ and $n = |E(M)|$.
Then we can compute a directed circuit basis $\mathcal{B}$ of $\vec{M}$ together with a set $J \subseteq E(\vec{M})$ such that $|J \cap B| \equiv 1 (\text{mod }2)$ for every $B \in \mathcal{B}$ in time polynomial in $mn$.
\end{lemma}

\begin{proof}
We shall follow the inductive proof of \Cref{lemma:circuitspace} to obtain a recursive algorithm for finding a desired directed circuit basis together with the desired set $J$.
If $n = 1$, the unique element $e$ of $E(\vec{M})$ is a directed loop, since $\vec{M}$ is totally cyclic, and forms our desired directed circuit basis of $\vec{M}$.
Furthermore, by setting $J :=  \{ e \}$ we also get our desired set.

In the case $n \ge 2$, let us fix an arbitrary element $e$ of $E(\vec{M})$ and compute a directed circuit $C_e$ of $\vec{M}$ containing $e$ by applying \Cref{lemma:farkas-compute}.
Also using \Cref{lemma:farkas-compute}, we can test in time polynomial in $mn$ whether $\vec{M} - e$ is totally cyclic.
If so, we fix $C_e$ as an element of our desired directed circuit base $\mathcal{B}$ of $\vec{M}$ and proceed as before with $\vec{M}-e$ instead of $\vec{M}$.
The set $J$ is updated as follows:
Suppose we have already computed a directed circuit base $\mathcal{B}^-$ and a set $J^-$ as in the statement of this lemma, but with respect to $\vec{M} - e$.
Then we set $\mathcal{B} :=  \mathcal{B}^- \cup \{ C_e \}$.
Now we check the parity of $|J^- \cap C_e|$ and set
\begin{align*}
J := 
\begin{cases}
J^-  & \mbox{if } |J^- \cap C_e| \equiv 1 \; (\text{mod }2) \\
J^- \cup \{ e \}  & \mbox{if } |J^- \cap C_e| \equiv 0 \; (\text{mod }2).
\end{cases}
\end{align*}
As $C_e$ is the only element of $\mathcal{B}$ that contains $e$, the set $J$ has odd intersection with every element of $\mathcal{B}$, as desired.

If $\vec{M} - e$ is not totally cyclic, we compute a totally unimodular representative matrix ${A' \in \{-1,0,1\}^{m \times (n-1)}}$ of $\vec{M}/e$.
This task can be executed in time polynomial in $mn$ \footnote{To compute $A'$, select a non-zero entry in the column of $A$ belonging to the element $e$. Pivoting on this element and exchanging rows transforms $A$ in polynomial time in $mn$ into a totally unimodular matrix $A'' \in \{-1,0,1\}^{m \times n}$ of $\vec{M}$ in which the column corresponding to the element $e$ of $\vec{M}$ is $(1,0,\ldots,0)^\top$. Then $\vec{M}[A]=\vec{M}[A'']$, and the matrix $A'$ obtained from $A''$ by deleting the first row is a totally unimodular representation of $\vec{M}/e$.}. 
Now $\vec{M}/e$ is totally cyclic as $\vec{M}$ is totally cyclic and we proceed as before with $\vec{M}/e$ instead of~$\vec{M}$.
However, when our recursive algorithm already yields a directed circuit basis $\mathcal{B}^-$ of $\vec{M}/e$ as well as a set $J^-$ for $\vec{M}/e$ as in the statement of this lemma, we know as argued in the proof of \Cref{lemma:circuitspace} that each element $C$ of $\mathcal{B}^-$ either is a directed circuit of $\vec{M}$ or $C \cup \{ e \}$ is a directed circuit of~$\vec{M}$.
Depending on this distinction we define our desired circuit basis $\mathcal{B}$ of $\vec{M}$ as in the proof of \Cref{lemma:circuitspace} via 
\begin{align*}
\mathcal{B}:= \{C \; | \; C \in \mathcal{B}^-\text{ circuit in }M\}\cup \{C \cup \{e\} \; | \; C \in \mathcal{B}^-, C \cup \{e\}\text{ circuit in }M\}.
\end{align*}
To decide for each element $C \in \mathcal{B}^-$ whether $C$ or $C \cup \{e\}$ is a directed circuit of $\vec{M}$ we calculate $A \mathbf{1}_C$ where $\mathbf{1}_C$ denotes the incidence vector of $C$ with respect to $A$.
Then $C$ forms a directed circuit of $\vec{M}$ if and only if $A \mathbf{1}_C = 0$.
As $|\mathcal{B}^-| = |\mathcal{B}| = |E(\vec{M})| - r(\vec{M})$ as argued in the proof of \Cref{lemma:circuitspace} and by \Cref{proposition:basesize}, we have to do at most $n$ of these computations to compute $\mathcal{B}$ from $\mathcal{B}^-$.
Regarding the set $J$ we can simply set $J :=  J^-$.
\end{proof}

We are now ready for the proof of \Cref{thm:oddjoinsandevencuts}.

\begin{proof}[Proof of \Cref{thm:oddjoinsandevencuts}]
Assume first we have access to an oracle deciding whether an oriented regular matroid given by a representing totally unimodular matrix is non-even.
Suppose we are given a regular oriented matroid $\vec{M}$ represented by a totally unimodular matrix $A \in \{-1,0,1\}^{m \times n}$ for some $m, n \in \mathbb{N}$ and we want to decide whether it contains a directed circuit of even size.

First we compute $TC(\vec{M})$, which can be done in time polynomial in $mn$ by \Cref{cor:totallycyclic}.
Now we use \Cref{lemma:compute_base_and_cover} to compute a directed circuit basis of $TC(\vec{M})$ in time polynomial in~$mn$.
Then we go through the $|E(TC(\vec{M}))|-r(TC(\vec{M}))$ many elements of the basis and check whether one of these directed circuits has even size.
If so, the algorithm terminates.
Otherwise, every member of the basis has odd size.
By \Cref{proposition:matroidequivalence} with $J :=  E(TC(\vec{M}))$, we know that $TC(\vec{M})$ contains no directed circuit of even size if and only if $TC(\vec{M})$ is non-even.
Since $TC(\vec{M})$ is the largest deletion minor of $\vec{M}$, which has the same directed circuits as $\vec{M}$, we know that $TC(\vec{M})$ is non-even if and only if $\vec{M}$ is non-even.
So we can decide the question using the oracle.

Conversely, assume we have access to an oracle which decides whether a given oriented regular matroid contains a directed circuit of even size.
Again, our first step is to compute $TC(\vec{M})$ using \Cref{cor:totallycyclic}.
By \Cref{lemma:compute_base_and_cover} we then compute a directed circuit basis of $TC(\vec{M})$ and a set $J \subseteq E(TC(\vec{M}))$ such that every circuit in the basis has odd intersection with $J$.

Let $\vec{M}'$ be the oriented matroid obtained from $TC(\vec{M})$ by duplicating every element $e \in E(TC(\vec{M})) \setminus J$ into two copies $e_1$ and $e_2$\footnote{In particular, we transform every signed circuit of $TC(\vec{M})$ into a signed circuit of $\vec{M}'$ by replacing every occurrence of an element $e \in E(TC(\vec{M})) \setminus J$ in a signed partition by the two elements $e_1, e_2$ in the same set of the signed partition.
It is not hard to see that this indeed defines an oriented matroid, which is still regular.}. This way, every directed circuit in $\vec{M}'$ intersects $E(\vec{M}') \setminus J$ in an even number of elements. Thus, for every directed circuit $C$ in $TC(\vec{M})$, the size of the corresponding directed circuit in $\vec{M}'$ is odd if and only if $|C \cap J|$ is odd.
Hence, $J$ intersects every directed circuit in $TC(\vec{M})$ an odd number of times if and only if $\vec{M}'$ contains no even directed circuit.
By \Cref{proposition:matroidequivalence} this shows that $TC(\vec{M})$ is non-even if and only if $\vec{M}'$ has no directed circuit of even size.
Since $TC(\vec{M})$ is non-even if and only if $\vec{M}$ is non-even, we can decide the non-evenness of $\vec{M}$ by negating the output of the oracle with instance $\vec{M}'$.
\end{proof}

With the tools developed in this section at hand we are ready for the proof of Proposition~\ref{prop:oddcycleproblem}.

\begin{proposition}\label{prop:oddcycleproblem}
There is an algorithm which given as input a totally unimodular matrix $A \in \mathbb{R}^{m \times n}$ for some $m, n \in \mathbb{N}$, either returns an odd directed circuit of $\vec{M}[A]$ or concludes that no such circuit exists, and runs in time polynomial in $mn$.
\end{proposition}
\begin{proof}
Let $A \in \mathbb{R}^{m \times n}$ be a totally unimdoular matrix given as input and let $\vec{M} := \vec{M}[A]$. To decide whether $\vec{M}$ contains a directed circuit of odd size, we first use \Cref{cor:totallycyclic} to compute a totally unimdoular representation of $TC(\vec{M})$ in polynomial time in $mn$. We now apply \Cref{lemma:compute_base_and_cover} to compute in polynomial time a directed circuit basis $\mathcal{B}$ of $TC(\vec{M})$. Going through the elements of $\mathcal{B}$ one by one, we test whether one of the basis-circuits is odd, in which case the algorithm stops an returns this circuit. Otherwise, all circuits in $\mathcal{B}$ are even. Since every circuit in the underlying matroid of $TC(\vec{M})$ can be written as a symmetric difference of elements of $\mathcal{B}$, every circuit in this matroid must be even. In particular, $TC(\vec{M})$ and hence $\vec{M}$ do not contain any odd directed circuits, and the algorithm terminates with this conclusion. 
\end{proof}

\section{Digraphs Admitting an Odd Dijoin}\label{sec:dijoin}

This section is dedicated to the proof of our main result, \Cref{thm:forbiddencographicminors}.
The overall strategy to achieve this goal is to work on digraphs and their families of bonds directly.
The object that certifies that the bond matroid of a digraph is non-even is called an \emph{odd dijoin}.

\begin{definition}
Let $D$ be a digraph. 
A subset $J \subseteq E(D)$ is called an \emph{odd dijoin} if $|J \cap S|$ is odd for every directed bond $S$ in $D$.
\end{definition}

Let $D$ be a digraph. 
The \emph{contraction} $D/A$ of an edge set $A \subseteq E(D)$ in $D$ is understood as the digraph arising from $D$ by deleting all edges of $A$ and identifying each weak connected component of $D[A]$ into a corresponding vertex. 
Note that this might produce new loops arising from edges spanned between vertices incident with $A$ but not included in $A$. Note that contracting a loop is equivalent to deleting the loop.

An edge $e=(x,y)$ of a digraph $D$, which is not a loop, is said to be \emph{deletable} (or \emph{transitively reducible}) if there is a directed path in $D$ starting in $x$ and ending in $y$ which does not use $e$. Note that an edge $e \in E(D)$ is deletable if and only if $e$ is a butterfly-contractible element of $M^\ast(D)$.

For two digraphs $D_1,D_2$, we say that $D_1$ is a \emph{cut minor} of $D_2$ if it can be obtained from $D_2$ by a finite series of edge contractions, deletions of deletable edges, and deletions of isolated vertices.

Our next lemma guarantees that the property of admitting an odd dijoin is closed under the cut minor relation.

\begin{lemma} \label{lemma:maintainment}
Let $D_1, D_2$ be digraphs such that $D_1$ is a cut minor of $D_2$.
If $D_2$ admits and odd dijoin, then so does $D_1$. 
\end{lemma}

\begin{proof}
The statement follows by applying \Cref{prop:GBminorcontainment} to $M^\ast(D_1)$ and $M^\ast(D_2)$, noting that deleting isolated vertices from a digraph does not change the induced oriented bond matroid.
\end{proof}

Our goal will be to characterise the digraphs admitting an odd dijoin in terms of forbidden cut minors. In the following, we prepare this characterisation by providing a set of helpful statements. For an undirected graph $G$, we define the \emph{cutspace} of $G$ as the $\mathbb{F}_2$-linear vector space generated by the bonds in $G$, whose addition operation is the symmetric difference and whose neutral element is the empty set. The following statements are all obtained in a striaghforward way by applying the oriented matroid results \Cref{lemma:circuitspace}, \Cref{cor:matroidweighting} respectively \Cref{proposition:matroidequivalence} to the oriented bond matroid $M^{\ast}(D)$ induced by $D$.

\begin{corollary} \label{cor:cutspace}
Let $D$ be a weakly connected and acyclic digraph with underlying multi-graph $G$.
Then the cut space of $G$ admits a basis $\mathcal{B}$ whose elements are the edge sets of minimal directed cuts in $D$. 
Moreover, if $A \subseteq E(D)$ is a set of edges such that $D/A$ is acyclic and $G[A]$ is a forest, then one can choose $\mathcal{B}$ such that every edge $e \in A$ appears in exactly one cut of the basis. \qed
\end{corollary}

\begin{corollary} \label{cor:weighting}
Let $D$ be a digraph and let $\mathcal{B}$ be a basis of the cut space consisting of minimal directed cuts. Then there is an edge set $J' \subseteq E(D)$ such that $|J' \cap B|$ is odd for all $S \in \mathcal{B}$. \qed
\end{corollary}

\begin{proposition} \label{proposition:equivalence}
 Let $D$ be a digraph, $\mathcal B$ be a basis of the cut space consisting of directed bonds, and let $J' \subseteq E(D)$ be such that $|B \cap J'|$ is odd for all $B \in \mathcal B$.
 Then the following statements are equivalent:
 \begin{enumerate}[label=(\roman*)]
  \item $D$ has an odd dijoin.
  \item If $B_1,\dotsc,B_k$ are directed bonds of $D$ with $k$ odd, then $\sum_{i=1}^k B_i \neq \emptyset$.
  \item Every directed bond of $D$ can be written as $\sum_{i=1}^k B_i$ with $k$ odd.
  \item $J'$ is an odd dijoin of $D$. \qed
 \end{enumerate}
\end{proposition}

\subsection{Forbidden cut minors for digraphs with an odd dijoin}
${}$
\vspace{6pt} \newline
\indent Next we characterise the digraphs admitting an odd dijoin in terms of forbidden cut minors.
For this purpose, we identify the digraphs without an odd dijoin for which every proper cut minor has an odd dijoin.
We call such a digraph a \emph{minimal obstruction}.
A digraph~$D=(V,E)$ is said to be \emph{oriented} if it has no loops, no parallel, and no anti-parallel edges.
Furthermore, $D$ is called \emph{transitively reduced} if for every edge $e = (v,w) \in E$ the only directed path in $D$ starting at $v$ and ending in $w$ consists of $e$ itself, or equivalently, if no edge in $D$ is deletable.

We start with the following crucial lemma, which will be used multiple times to successively find the structure minimal obstructions must have.

\begin{lemma} \label{lemma:basicproperties}
Let $D$ be a minimal obstruction.
Then the underlying multi-graph $G$ of $D$ is 2-vertex-connected.
Furthermore, $D$ is oriented, acyclic, and transitively reduced.
\end{lemma}

\begin{proof}
Assume that $D$ has no odd dijoin, but every cut minor of $D$ has one.
Then it is easy to check that $|V(D)| \ge 4$.

To prove that $G$ must be $2$-vertex-connected, suppose towards a contradiction that $G$ can be written as the union of two proper subgraphs $G_1, G_2$ with the property that $|V(G_1) \cap V(G_2)| \le 1$.
Then the orientations $D_1,D_2$ induced on $G_1, G_2$ by $D$ are proper cut minors of $D$:
Indeed, for $i \in \{1,2\}$ we can obtain $D_i$ from $D$ by contracting all edges in $D_{2-i}$ and then deleting all the resulting isolated vertices outside $V(D_i)$.
Since $D_1, D_2$ are proper cut minors of $D$, they must admit odd dijoins $J_1,J_2$, respectively.
However, since $D_1$ and $D_2$ share at most a single vertex, the directed bonds of $D$ are either directed bonds of $D_1$ or of $D_2$.
Hence, the disjoint union $J_1 \cup J_2$ defines an odd dijoin of $D$ and yields the desired contradiction.

To prove acyclicity, assume towards a contradiction that there is a directed cycle $C$ in~$D$.
Let us consider the digraph $D/E(C)$.
This is a proper cut minor of $D$ and therefore must have an odd dijoin $J$.
However, the directed bonds in $D/E(C)$ are the same as the directed bonds in $D$ edge-disjoint from $C$, and since $C$ is directed, these are already all the directed bonds of $D$. Hence $J$ is an odd dijoin also for~$D$, which is a contradiction.

To prove that $D$ is transitively reduced, assume towards a contradiction that there was an edge $e=(x,y) \in E(D)$ and a directed path $P$ from $x$ to $y$ not containing $e$. 
Then $e$ is a deletable edge and $D-e$ is a cut minor of $D$, which therefore must have an odd dijoin $J \subseteq E(D) \setminus \{e\}$. 
Note that a directed cut in $D$ either does not intersect $\{e\} \cup E(P)$ at all or contains $e$ and exactly one edge from $P$.
It follows from this that for every directed bond $B$ in $D$,  we get that $B-e$ is a directed bond of $D-e$.
This directly yields that $J$ is also an odd dijoin of $D$, contradiction.

Clearly, the fact that $D$ is oriented follows from $D$ being simultaneously acyclic and transitively reduced.
This concludes the proof of the lemma.
\end{proof}

From this, we directly have the following useful observations.

\begin{corollary} \label{cor:contract}
Let $D$ be a minimal obstruction.
Then for every edge $e \in E(D)$, the digraph $D/e$ is acyclic.
Similarly, for every vertex $v \in V(D)$ which is either a source or a sink, the digraph $D / E(v)$, with $E(v) :=  D[ \{ v \},V(D) \setminus \{ v \}]$, is acyclic.
\end{corollary}

\begin{proof}
Let $e$ be an edge of $D$.
Since $D$ is a minimal obstruction, we know by \Cref{lemma:basicproperties} that $e$ is no loop.
Now assume towards a contradiction that there was a directed cycle in $D/e$.
As $D$ itself is acyclic according to \Cref{lemma:basicproperties}, this implies that there is a directed path $P$ in $D$ connecting the end vertices of $e$, which does not contain $e$ itself.
This path together with $e$ now either contradicts the fact that $D$ is acyclic or the fact that $D$ is transitively reduced, both of which hold due to \Cref{lemma:basicproperties}.

For the second part assume w.\@l.\@o.\@g.\@ (using the symmetry given by reversing all edges) that $v$ is a source.
Suppose for a contradiction there was a directed cycle in $D/E(v)$.
This implies the existence of a directed path $P$ in $D-v$ which connects two different vertices in the neighbourhood of $v$, say it starts in $w_1 \in N(v)$ and ends in $w_2 \in N(v)$.
Now the directed path $(v,w_1)+P$ witnesses that the directed edge $(v,w_2)$ is deletable contradicting that $D$ is transitively reduced. 
This concludes the proof of the second statement. 
\end{proof}

\begin{lemma} \label{lemma:acycliccontract}
Let $D$ be a minimal obstruction.
If $A \subseteq E(D)$ is such that $D/A$ is acyclic and such that $D[A]$ is an oriented forest, then there is a directed bond in $D$ which fully contains $A$.
\end{lemma}

\begin{proof}
By \Cref{cor:cutspace} there is a basis $\mathcal{B}$ of the cut space consisting of directed bonds such that each $e \in A$ is contained in exactly one of the bonds in the basis.
Moreover, by \Cref{cor:weighting} there is $J' \subseteq E(D)$ such that each $B \in \mathcal{B}$ has odd intersection with~$J'$.
Since $D$ has no odd dijoin, there has to be a directed bond $B_0$ in $D$ such that $|B_0 \cap J'|$ is even.
Let $B_0=B_1+\ldots+B_m$ be the unique linear combination with pairwise distinct $B_1,\ldots,B_m \in \mathcal{B}$.
Clearly, $m$ must be even.
Let $D'$ be the cut minor obtained from $D$ by contracting the edges in $E(D) \setminus \bigcup_{i=1}^{m}{B_i}$.
The bonds $B_0,B_1,\ldots,B_m$ are still directed bonds in $D'$ and satisfy $B_0+\ldots+B_m=\emptyset$, while $m+1$ is odd.
The equivalence of (i) and (ii) in \Cref{proposition:equivalence} now yields that $D'$ has no odd dijoin.
By the minimality of $D$ we thus must have $D=D'$ and $\bigcup_{i=1}^{m}{B_i}=E(D)$.
It follows that every $e \in A$ is contained in exactly one of the bonds $B_i$ and thus also in $B_0$.
Therefore, $B_0 \supseteq A$.
\end{proof}

\begin{corollary} \label{no separating directed cut}
Let $D = (V, E)$ be a minimal obstruction.
For $i \in \{1,2\}$ let $\emptyset \neq A_i \subseteq E$ be such that $D[A_i]$ is a forest and $D/A_i$ is acyclic.
Suppose there is a directed cut~$\Cut{}{X}$ in $D$ separating $A_1$ from $A_2$, i.\@e.\@, such that $A_1 \subseteq E(D[X])$ and $A_2 \subseteq E(D[V \setminus X])$. Then there exists a directed bond in $D$ containing $A_1 \cup A_2$.
\end{corollary}

\begin{proof}
 Let $A :=  A_1 \mathbin{\dot\cup} A_2$.
 As $A_1$ and $A_2$ induce vertex-disjoint forests, $D[A]$ is a forest as well.
 Since no edge is directed from a vertex in $V \setminus X$ to a vertex in $X$, no directed circuit in $D/A$ can contain a contracted vertex from $A_1$ and a contracted vertex from $A_2$, so every directed circuit must already exist in $D/A_1$ or in $D/A_2$.
 Because these two digraphs are acyclic, $D/A$ is acyclic.
 Hence, by \Cref{lemma:acycliccontract}, $A$ is fully included in a directed bond of $D$. This proves the assertion.
\end{proof}

With the next proposition we shall make the structure of minimal obstructions much more precise.
To state the result, we shall make use of the following definition.

\begin{definition}
 Let $n_0, n_1, n_2 \in \mathbb{N}$.
 Then we denote by $\mathcal D(n_0,n_1,n_2)$ the digraph $(V,E)$, where $V = V_0 \mathbin{\dot\cup} V_1 \mathbin{\dot\cup} V_2$ with $V_i = [n_i]$ for $i \in \{1,2,3\}$, and $E = (V_0 \times V_1) \mathbin{\dot\cup} (V_1 \times V_2)$.
\end{definition}

\begin{proposition} \label{prop:threenumberreduction}
Let $D = (V, E)$ be a minimal obstruction.
Then $D$ is isomorphic to $\mathcal{D}(n_1,n_2,n_3)$ for some integers $n_1,n_2,n_3 \ge 0$.
\end{proposition}

\begin{proof}
We shall split the proof into several claims, starting with the following one.

\begin{claim}\label{claim:length3}
$D$ contains no directed path of length $3$.
\end{claim}

Suppose towards a contradiction that $v_0,e_1,v_1,e_2,v_2,e_3,v_3$ is a directed path of length $3$ in $D$ with $e_1=(v_0,v_1),e_2=(v_1,v_2), e_3=(v_2,v_3)$.
By \Cref{cor:contract}, $D/e_1$ and $D/e_3$ are acyclic.
Moreover, because $D$ is acyclic by \Cref{lemma:basicproperties}, the edge $e_2$ is contained in a directed cut~$\Cut{}{X}$ in $D$, separating $\{e_1\}$ and $\{e_3\}$.
By \Cref{no separating directed cut} this means that there is a directed bond $\partial(Y)$ in $D$ containing both $e_1$ and $e_2$.
This however means that $v_0, v_2  \in Y$ and $v_1, v_3 \notin Y$.
Hence, $e_2$ is an edge in $D$ starting in $V(D) \setminus Y$ and ending in $Y$, a contradiction since $\partial(Y)$ is a directed bond.
This completes the proof of \Cref{claim:length3}.
\vspace{6pt}

For $i \in \{ 0, 1, 2 \}$ let $V_i$ denote the set of vertices $v \in V$ such that the longest directed path ending in $v$ has length~$i$.
By definition of the $V_i$ and since $D$ is acyclic, there is no edge from a vertex in $V_i$ to a vertex in $V_j$ for $i \ge j$, as otherwise this would give rise to a directed path of length $i+1$ ending in a vertex of $V_j$.

By \Cref{claim:length3} we know that $V = V_0 \mathbin{\dot\cup} V_1 \mathbin{\dot\cup} V_2$ holds.
We move on by proving the following claim.

\begin{claim}\label{claim:all edges 0-1}
Every vertex~$v \in V_1$ is adjacent to every vertex~$u \in V_0$.
\end{claim}

Let $v \in V_1$ and $u \in V_0$.
Assume for a contradiction that $u$ is not adjacent to $v$.
By definition of $V_1$ there is an edge $f = (u',v)$ with $u' \in V_0$.
By \Cref{cor:contract}, $D/f$ and $D/E(u)$ are acyclic because $u$ is a source.
Let $X \supseteq \{u',v\}$ be the set of all vertices from which $v$ can be reached via a directed path.
Clearly $\Cut{}{X}$ is a directed cut in $D$.
As $u \in V_0 \setminus X$ is a source, we conclude that $\{u\} \cup N(u) \subseteq V \setminus X$.
This however means that the directed cut $\Cut{}{X}$ separates $f$ from the edges in $E(u)$.
By \Cref{no separating directed cut}, this means that there is a directed bond $\partial(Y)$ in $D$ containing $E(u) \cup \{f\}$.
Since $E(u) = \Cut{}{\{u\}}$ itself is a directed cut in $D$, this contradicts the fact that $\partial(Y)$ is an inclusion-wise minimal directed cut in $D$, and proves \Cref{claim:all edges 0-1}.
\vspace{6pt}

We proceed with another claim.

\begin{claim}\label{claim:no 2-layer edges}
$D$ does not contain any edge from $V_0$ to $V_2$.
\end{claim}

Let $u \in V_0$ and $w \in V_2$.
By definition of $V_2$ there is some $v \in V_1$ such that $(v,w) \in E$.
By \Cref{claim:all edges 0-1}, $(u,v) \in E$.
Because $D$ is transitively reduced by \Cref{lemma:basicproperties}, we obtain $(u,w) \notin E$.
So the proof of \Cref{claim:no 2-layer edges} is complete.
\vspace{6pt}

Now we come to the last claim we need for the proof of this proposition.

\begin{claim}
Every vertex~$v \in V_1$ is adjacent to every vertex~$w \in V_2$.
\end{claim}

Let $v \in V_1, w \in V_2$ and suppose for a contradiction that $w$ is not adjacent to $v$.
Let $f = (u,v)$ be an edge with $u \in V_0$.
By \Cref{lemma:acycliccontract}, $D/f$ and $D/E(w)$ are acyclic because $w$ is a sink.
Let $X \supseteq \{u,v\}$ be the set of all vertices from which $v$ can be reached via a directed path.
Again, $\Cut{}{X}$ forms a directed cut in $D$.
\Cref{claim:no 2-layer edges} implies that $N(w) \subseteq V_1 \setminus \{v\} \subseteq V \setminus X$.
This means $\Cut{}{X}$ separates $f$ from the edges in $E(w)$, contradicting \Cref{no separating directed cut} again.
\vspace{6pt}

By combining all four claims we obtain $E=(V_0 \times V_1) \dot{\cup} (V_1 \times V_2)$, and the proof of this proposition is complete.
\end{proof}

Now \Cref{prop:threenumberreduction} puts us in the comfortable situation that the only possible minimal obstructions to having an odd dijoin are part of a $3$-parameter class of simply structured digraphs.
The rest of this section is devoted to determine the conditions on $n_1,n_2,n_3$ that need to be imposed such that $\mathcal{D}(n_1,n_2,n_3)$ is a minimal obstruction.
It will be helpful to use the well-known concept of so-called $T$-joins.

\begin{definition}
Let $G$ be an undirected graph and $T \subseteq V(G)$ be some vertex set.
A subset $J \subseteq E(G)$ of edges is called a \emph{$T$-join}, if in the subgraph $H:= G[J]$ of $G$, every vertex in $T$ has odd, and every vertex in $V(G) \setminus T$ has even degree.
\end{definition}

The following result is folklore.

\begin{lemma} \label{lemma:tjoin}
A graph $G$ with some vertex set $T \subseteq V(G)$ admits a $T$-join if and only if $T$ has an even number of vertices in each connected component of $G$.
\end{lemma}

We continue with an observation about odd dijoins in digraphs of the form $\mathcal{D}(n_1,n_2,0)$.

\begin{observation} \label{observation:obs}
Let $n_1,n_2 \ge 1$.
Then the digraph $\mathcal{D}(n_1,n_2,0) \simeq \mathcal{D}(0,n_1,n_2)$ has an odd dijoin if and only if $\text{min}(n_1,n_2) \leq 1$ or $n_1,n_2 \ge 2$ and $n_1 \equiv n_2 \text{ (mod }2)$.
\end{observation}

\begin{proof}
If $\text{min}(n_1,n_2) \leq 1$, then all directed bonds in $\mathcal{D}(n_1,n_2,0)$ consist of single edges and thus, $J:= E(\mathcal{D}(n_1,n_2,0))$ defines an odd dijoin.
If $n_1,n_2 \ge 2$, the directed bonds in $\mathcal{D}(n_1,n_2,0)$ are exactly those cuts with one vertex on one side of the cut and all other vertices on the other side.
Hence, there is an odd dijoin if and only if the complete bipartite graph with partition classes of size $n_1,n_2$ has a $T$-join, where $T$ contains all $n_1+n_2$ vertices.
The statement is now implied by \Cref{lemma:tjoin}.
\end{proof}

Next we characterise when the digraphs $\mathcal{D}(n_1,n_2,n_3)$ admit an odd dijoin.

\begin{proposition} \label{proposition:threenumbers}
Let $n_1,n_2,n_3 \ge 1$ be integers.
Then $\mathcal{D}(n_1,n_2,n_3)$ has an odd dijoin if and only if one of the following holds:
\begin{enumerate}[label=(\roman*)]
\item $n_2=1$.
\item $n_2=2$ and $n_1 \equiv n_3 \text{ (mod }2)$.
\item $n_2 \ge 3$, and $n_1 \equiv n_3 \equiv 1 \text{ (mod }2)$.
\end{enumerate}
\end{proposition}

\begin{proof}
If $n_2=1$, then $\mathcal{D}(n_1,n_2,n_3)$ is an oriented star.
Clearly, here, the directed bonds consist of single edges, and therefore, $J :=  E(\mathcal{D}(n_1,1,n_3))$ defines an odd dijoin.

If $n_2=2$, it is easily seen that $\mathcal{D}(n_1,2,n_3)$ is a planar digraph, which admits a directed planar dual isomorphic to a bicycle $\Bidirected{C}_{n_1+n_3}$ of length $n_1+n_3$.
By planar duality, we know that $\mathcal{D}(n_1,2,n_3)$ has an odd dijoin if and only if there is a subset of edges of $\Bidirected{C}_k$ which intersects every directed cycle an odd number of times.
By \Cref{thm:seymthomtheorem} we know that such an edge set exists if and only if $n_1+n_3$ is even, that is, $n_1 \equiv n_3 \text{ (mod }2)$.

Therefore, we assume that $n_2 \ge 3$ for the rest of the proof.
We now first show the necessity of~(iii).
So assume that $D :=  \mathcal{D}(n_1,n_2,n_3)$ has an odd dijoin $J$.
We observe that the underlying multi-graph of $D$ is $2$-connected.
Hence, for every vertex $x \in V_1 \cup V_2 \cup V_3$, the cut $E(x)$ of all edges incident with $x$ is a minimal cut of the underlying multi-graph, and it is directed whenever $x \in V_1 \cup V_3$.
Therefore, $U(D[J])$ must have odd degree at every vertex in $V_1 \cup V_3$.
Moreover, we observe that for any proper non-empty subset $X \subsetneq V_2$, the cut in $D$ induced by the partition $(V_1 \cup X, (V_2 \setminus X) \cup V_3)$ is minimal and directed.
In the following, we denote this cut by $F(X)$.
Now for every vertex $x \in V_2$, choose some $x' \in V_2 \setminus \{x\}$ and consider the minimal directed cuts $F(\{x'\}), F(\{x,x'\})$.
Both are minimal directed cuts (here, we use that $n_2 \ge 3$) and thus must have odd intersection with~$J$.
Moreover, the symmetric difference $F(\{x'\})+F(\{x,x'\})$ contains exactly the set $E(x)$ of edges incident with $x$ in $D$.
We conclude the following:
\begin{align*}
|E(x) \cap J|=|(F(\{x'\})+F(\{x,x'\})) \cap J| & \equiv |F(\{x'\}) \cap J|+|F(\{x,x'\}) \cap J| \\
& \equiv 1+1 \equiv 0 \text{ (mod }2)
\end{align*}
As $x \in V_2$ was chosen arbitrarily, we conclude that $J$ must be a $T$-join of the underlying multi-graph of $\mathcal{D}(n_1,n_2,n_3)$ where $T = V_1 \cup V_3$.
Now \Cref{lemma:tjoin} implies that ${|T| = n_1+n_3}$ must be even and hence $n_1 \equiv n_3 \text{ (mod }2)$.

We claim that (iii) must be satisfied, i.\@e.\@, $n_1$ and $n_3$ are odd. 
Assume towards a contradiction that this is not the case.
Hence, by our observation above both $n_1$ and $n_3$ are even.
Let $x \in V_2$ be some vertex, and consider the directed bond $F(\{x\})$.
We can rewrite this bond as the symmetric difference of the directed cut ${\Cut{}{V_1}=\{(v_1,v_2) \; | \; v_1 \in V_1, v_2 \in V_2\}}$ and the cut $E(x)$ of all edges incident with $x$.
Because $|E(u) \cap J|$ is odd for every $u \in V_1$, we obtain that $|\Cut{}{V_1} \cap J|=\sum_{u \in V_1}{|E(u) \cap J|}$ must be even.
However, since also $|E(x) \cap J|$ is even, this means that $|F(\{x\}) \cap J| \equiv |\Cut{}{V_1} \cap J|+|E(x) \cap J| \equiv 0 \text{ (mod }2)$, which is the desired contradiction, as $J$ is an odd dijoin.
So (iii) must be satisfied.

To prove the reverse direction, assume that (iii) is fulfilled, i.\@e.\@, $n_1 \equiv n_3 \equiv 1 \text{ (mod }2)$.
We shall construct an odd dijoin of $\mathcal{D}(n_1,n_2,n_3)$.
For this purpose, we choose $J$ to be a $T$-join of the underlying multi-graph where $T=V_1 \cup V_3$.
We claim that this defines an odd dijoin of $\mathcal{D}(n_1,n_2,n_3)$.
It is not hard to check that the directed bonds of $\mathcal{D}(n_1,n_2,n_3)$ are the cuts $E(v)$ for vertices $v \in V_1 \cup V_3$ and the cuts $F(X)$ as described above, where $\emptyset \neq X \subsetneq V_2$.
By the definition of a $T$-join, all of the directed bonds of the first type have an odd intersection with $J$, so it suffices to consider the bonds of the second type.
Consider again the directed cut $\Cut{}{V_1}$ in $\mathcal{D}(n_1,n_2,n_3)$.
For any $\emptyset \neq X \subsetneq V_2$, we can write $F(X)$ as the symmetric difference $F(X)=\Cut{}{V_1}+\sum_{x \in X}{E(x)}$.
We therefore conclude that
\begin{align*}
|F(X) \cap J| & \equiv |\Cut{}{V_1} \cap J| + \sum_{x \in X}{\underbrace{|E(x) \cap J|}_{\text{even}}} \text{ (mod }2) \\
& \equiv |\Cut{}{V_1} \cap J| = \sum_{x \in V_1}{\underbrace{|E(x) \cap J|}_{\text{odd}}} \equiv n_1 \equiv 1 \text{ (mod }2).
\end{align*}
This verifies that $J$ is an odd dijoin, and completes the proof of the proposition.
\end{proof}

We shall now use these insights to characterise minimal obstructions.
For this let us first introduce new notation.

Let $D$ be a digraph consisting of a pair $h_1,h_2$ of ``hub vertices'' and other vertices $x_1,\ldots,x_n$, where $n \ge 3$, such that for every $i \in [n]$, the vertex $x_i$ has either precisely two outgoing or precisely two incoming edges to both $h_1,h_2$, and these are all edges of $D$. 
In this case, we refer to $D$ as a \emph{diamond}. 
Furthermore, we call any digraph isomorphic to $\vec{K}_{n_1,n_2}$ for some $n_1,n_2 \ge 2$, a \emph{one-direction}.
We shall call both, diamonds and one-directions, \emph{odd} if the total number of vertices of these digraphs is odd.

\begin{lemma} \label{lemma:verifyobstructions}
All odd diamonds and all odd one-directions are minimal obstructions.
\end{lemma}

\begin{proof}
It is directly seen from \Cref{observation:obs} and \Cref{proposition:threenumbers} that indeed, odd diamonds and odd one-directions do not posses an odd dijoin.
Therefore it remains to show that all proper cut minors of these digraphs have odd dijoins.
Because both odd diamonds and odd one-directions are weakly $2$-connected, transitively reduced and acyclic, the only cut minor operation applicable to them in the first step is the contraction of a single edge.
By \Cref{lemma:maintainment} it therefore suffices to show that for both types of digraphs, the contraction of any edge results in a digraph admitting an odd dijoin.
We first consider odd diamonds.
Let $D=\mathcal{D}(n_1,2,n_3)$ with $n_1,n_2 \ge 1$ and $n_1+n_2$ odd, and let $e \in E(D)$ be arbitrary.
In the planar directed dual graph of $D$, an odd bicycle with $n_1+n_2$ vertices, there is a directed dual edge corresponding to $e$.
It is easily seen by duality that $D/e$ has an odd dijoin if and only if the odd bicycle of order $n_1+n_2 \ge 3$ with a single deleted edge has an edge set intersecting every directed cycle an odd number of times.
However, this is the case, because such a digraph is non-even by \Cref{thm:seymthomtheorem}.

Now we consider odd one-directions.
Let $D = \mathcal{D}(n_1,n_2,0)$ with $n_1,n_2 \ge 2$ and $n_1+n_2$ odd, and let $e=(x,y) \in E(D)$ be arbitrary.
Then in the digraph $D/e$, define $J$ to be the set of all edges incident with the contraction vertex.
It is easily observed that $J$ intersects every minimal directed cut exactly once and thus indeed, every proper cut minor has an odd dijoin.
This completes the proof.
\end{proof}

Now we are able to prove a dual version of \Cref{thm:seymthomtheorem} and characterise the existence of odd dijoins in terms of forbidden cut minors.

\begin{theorem}\label{thm:forbiddenminors}
A digraph admits an odd dijoin if and only if it does neither have an odd diamond nor an odd one-direction as a cut minor.
\end{theorem}

\begin{proof}
By \Cref{lemma:maintainment}, a digraph has an odd dijoin if and only if it does not contain a minimal obstruction as a cut minor.
Hence it suffices to show that a digraph $D$ is a minimal obstruction if and only if it is isomorphic to an odd diamond or an odd one-direction.
The fact that these digraphs indeed are minimal obstructions was proved in \Cref{lemma:verifyobstructions}.
So it remains to show that these are the only minimal obstructions.

Let $D$ be an arbitrary minimal obstruction.
By \Cref{prop:threenumberreduction} there are integers $n_1,n_2,n_3 \ge 0$ such that $D \simeq \mathcal{D}(n_1,n_2,n_3)$.
By the definition of a minimal obstruction, we know that $D$ has no odd dijoin, while for every edge $e \in E(D)$, the digraph $D/e$ is a cut minor of $D$ and therefore has one. 
We know due to \Cref{lemma:basicproperties} that $D$ is weakly $2$-connected.
Hence, we either have ${\text{min}(n_1,n_3) = 0}$, so (by symmetry) w.l.o.g. $n_3=0$, or $n_1,n_3 \ge 1$ and therefore $n_2 \ge 2$. 

In the first case, we know by \Cref{observation:obs} and using that $D$ has no odd dijoin, that $n_1,n_2 \ge 2$ and $n_1 \not\equiv n_2 \text{ (mod }2)$.
So $D$ is an odd one-direction, which verifies the claim in the case of $\text{min}(n_1,n_3)=0$. 

Next assume that $n_1,n_3 \ge 1$ and $n_2 \ge 2$.
Let $e=(x_1,x_2) \in E(D)$ with $x_i \in V_i$ for $i=1,2$ be an arbitrary edge going from the first layer $V_1$ to the second layer $V_2$.
Denote by $c$ the vertex of $D/e$ corresponding to the contracted edge $e$.
Then in the digraph $D/e$, all edges $\{(c,v_3) \; | \; v_3 \in V_3\}$ as well as all the edges in $\{(v_1,v_2) \; | \; v_1 \in V_1 \setminus \{x_1\},v_2 \in V_2 \setminus \{x_2\}\}$ admit parallel paths since $n_2 \ge 2$ and, therefore, are deletable.
Successive deletion yields a cut minor $D'$ of $D/e$, and thus of $D$, with vertex set
\begin{align*}
V(D') = (V_1\setminus \{x_1\}) \cup \{c\} \cup (V_2\setminus \{x_2\}) \cup V_3
\end{align*}
and edge set
\begin{align*}
E(D')\! =\! \{(v_1,c) \; | \; v_1 \in V_1\setminus\{x_1\}\} \!\cup\! \{(c,v_2) \; | \; v_2 \in V_2 \setminus\{x_2\}\} \!\cup\! \{(v_2,v_3) \; | \; v_2 \in V_2 \setminus \{x_2\}, v_3 \in V_3\}.
\end{align*}

Now after contracting all edges of $D'$ of the set $\{(v_1,c) \; | \; v_1 \in V_1\setminus\{x_1\}\}$ we find that $D'$, and hence $D$, has a proper cut minor isomorphic to $\mathcal{D}(1,n_2-1,n_3)$ with corresponding layers $\{c\}, V_2 \setminus \{x_2\}$ and $V_3$. 

Applying a symmetric argument (starting by contracting an edge going from $V_2$ to $V_3$), we find that $D$ also has a proper cut minor isomorphic to $\mathcal{D}(n_1,n_2-1,1)$. 

Using these insights, we now show that $n_2 = 2$ holds. 
Suppose for a contradiction that $n_2 \ge 3$ holds.
Assume first that $n_2 \ge 4$, and therefore $n_2 - 1 \ge 3$.
Using statement (iii) of \Cref{proposition:threenumbers} and that $\mathcal{D}(1,n_2-1,n_3)$ and $\mathcal{D}(n_1,n_2-1,1)$ both have odd dijoins, we must have $n_1 \equiv n_3 \equiv 1 \text{ (mod }2)$. 
In the case that $n_2=3$, we similarly observe from statement (ii) of \Cref{proposition:threenumbers} with the digraphs $\mathcal{D}(1,2,n_3)$ and $\mathcal{D}(n_1,2,1)$ that both $n_1$ and $n_3$ must be odd.
Now using statement (iii) of \Cref{proposition:threenumbers} with the digraph $D \simeq \mathcal{D}(n_1,n_2,n_3)$ we can conclude that $D$ must admit an odd dijoin as well, a contradiction.

Hence, we must have $n_2=2$.
Using again statement (ii) of \Cref{proposition:threenumbers} with $D \simeq \mathcal{D}(n_1,2,n_3)$, we get that $n_1+n_3$ must be odd.
Therefore $D$ is isomorphic to an odd diamond with $2+n_1+n_3$ many vertices.
This concludes the proof of the theorem.
\end{proof}

We are now ready to give the proof of \Cref{thm:forbiddencographicminors}.

\begin{proof}[Proof of \Cref{thm:forbiddencographicminors}]
Let $\vec{M}$ be an oriented bond matroid, and let $D$ be a digraph such that $\vec{M} \simeq M^\ast(D)$.
Let us first note that by definition, $\vec{M}$ is non-even if and only if $D$ has an odd dijoin.
Hence, for the equivalence claimed in this theorem it suffices to show that $D$ has an odd dijoin if and only if $M^\ast(D)$ does not have a \GB-minor isomorphic to $\vec{K}_{m,n}$ for $m, n \ge 2$ such that $m+n$ is odd.
Suppose first that $D$ has an odd dijoin and $M^\ast(D)$ is non-even. Then by \Cref{prop:GBminorcontainment}, every \GB-minor of $M^\ast(D)$ is non-even as well, and hence, no such minor can equal $M^\ast(\vec{K}_{m,n})$ for any $m, n \ge 2$ with $m+n$ is odd, since $\vec{K}_{m,n}$ does not have an odd dijoin for any such $m$ and $n$ by \Cref{lemma:verifyobstructions}.
This proves the first implication of the equivalence. 

Conversely, let us suppose that $M^\ast(D)$ does not have a \GB-minor isomorphic to $M^\ast(\vec{K}_{m,n})$ for any $m, n \ge 2$ such that $m+n$ is odd.
We shall show that $D$ admits an odd dijoin.
For this we use \Cref{thm:forbiddenminors} and verify that $D$ has neither an odd diamond nor an odd one-direction as a cut minor.
This however follows directly from the fact that the bond-matroid induced by any odd diamond of order $n$ is isomorphic to $M^\ast(\vec{K}_{2,n-2})$ as well as the easy observation that if $D'$ is a cut minor of $D$, then $M^\ast(D')$ is a \GB-minor of $M^\ast(D)$.
This finishes the proof of the claimed equivalence.
\end{proof}

\section{Concluding remarks}\label{sec:conclusion}
For every odd $k \ge 3$ it holds that $M(\Bidirected{C_k}) \simeq M^\ast(\vec{K}_{k,2}) \simeq M^\ast(\vec{K}_{2,k})$, and hence, the list of smallest excluded \GB-minors characterising non-evenness for cographic oriented matroids strictly extends the list for graphic ones.
We find this quite surprising and did not expect it when we initiated our research on the subject.

Seymour~\cite{seymregularsums} has proved a theorem about generating the class of regular matroids, showing that every regular matroid can be built up from graphic matroids, bond matroids and a certain $10$-element matroid $R_{10}$ by certain sum operations.
The matroid $R_{10}$ is regular, but neither graphic nor cographic.
It is given by the following totally unimodular representing matrix:
\begin{align*}
R_{10}=M\left[\left(\begin{matrix}
1 & 0 & 0 & 0 & 0 & -1 & 1 & 0 & 0 & 1 \cr
0 & 1 & 0 & 0 & 0 & 1 & -1 & 1 & 0 & 0 \cr
0 & 0 & 1 & 0 & 0 & 0 & 1 & -1 & 1 & 0 \cr
0 & 0 & 0 & 1 & 0 & 0 & 0 & 1 & -1 & 1 \cr
0 & 0 & 0 & 0 & 1 & 1 & 0 & 0 & 1 & -1 
\end{matrix} \right) \right]
.
\end{align*}

Seymour introduced three different kinds of sum operation which join two regular matroids $M_1$ and $M_2$ whose element sets are either disjoint (1-sum), intersect in a single non-loop element (2-sum) or in a common $3$-circuit (3-sum) into a bigger regular matroid $M_1 \Delta M_2$ (for a precise definition of these operations we refer to the introduction of \cite{seymregularsums}). 

\begin{theorem}[\cite{seymregularsums}]
Every regular matroid can be built up from graphic matroids, bond matroids and $R_{10}$ by repeatedly applying 1-sums, 2-sums and 3-sums.
\end{theorem}

This theorem shows that graphic matroids, bond matroids and $R_{10}$ constitute the most important building blocks of regular matroids.
Using a brute force implementation, we checked by computer that every orientation of $R_{10}$ containing no $M^\ast(\vec{K}_{m,n})$ as a \GB-minor for any $m, n \ge 2$ such that $m+n$ is odd, is already non-even. 
We therefore expect the total list of forbidden minors for all non-even oriented matroids to not be larger than the union of the forbidden minors for graphic (\Cref{thm:forbiddengraphicminors}) and cographic (\Cref{thm:forbiddencographicminors}) non-even oriented matroids. In other words, we conjecture the following.

\begin{conjecture}
	A regular oriented matroid $M$ is non-even if and only if none of its \GB-minors is isomorphic to $M^\ast(\vec{K}_{m,n})$ for some $m, n \ge 2$ such that $m+n$ is odd.
\end{conjecture}

The natural way of working on this conjecture would be to try and show that a smallest counterexample is not decomposable as the $1$-, $2$- or $3$-sum of two smaller oriented regular matroids.
Apart from the obvious open problem of resolving the computational complexity of the even circuit problem (\Cref{evencircuit1}) for regular oriented matroids in general, already resolving the case of bond matroids would be interesting.

\begin{problem}
Is there a polynomially bounded algorithm that, given as input a digraph $D$, decides whether or not $D$ contains a directed bond of even size?
Equivalently, is there a polynomially bounded recognition algorithm for digraphs admitting an odd dijoin?
\end{problem}

Conclusively, given our characterisation of digraphs admitting an odd dijoin in terms of forbidden cut minors, the following question naturally comes up.

\begin{problem}
Let $F$ be a fixed digraph.
Is there a polynomially bounded algorithm that, given as input a digraph $D$, decides whether or not $D$ contains a cut minor isomorphic to $F$?
\end{problem}

\section*{Acknowledgement}

We thank Winfried Hochst\"{a}ttler and Sven Jäger for discussions on the topic.

Karl Heuer was supported by the European Research Council (ERC) under the European Union's Horizon 2020 research and innovation programme (ERC consolidator grant DISTRUCT, agreement No.\ 648527).
Raphael Steiner was funded by DFG-GRK 2434.

\end{document}